\documentclass[12pt]{article}

\usepackage[margin=1in]{geometry} 
\usepackage{graphicx}            
\usepackage{amsfonts}  
\usepackage[breaklinks]{hyperref}
\usepackage{tikz}
\usepackage{float}
\usetikzlibrary{patterns,arrows,decorations.pathreplacing}
\usetikzlibrary{decorations.markings}
\usetikzlibrary{calc,shapes}
\usepackage{calc}

\newcommand{\s}[1]{\left\lvert #1 \right\rvert}
\newcommand{\br}[1]{\left( #1 \right)}

\newcommand{\eps}{\varepsilon}

\usepackage{amsmath,amsthm,amssymb,enumerate,mathrsfs,mathtools}
\usepackage{relsize}
\usepackage{a4wide}
\usepackage{verbatim}
\usepackage{xcolor}

\usepackage{makecell}
\usepackage[numbers,sort&compress]{natbib}

\usepackage{enumitem}
\setlistdepth{9} 
\renewlist{enumerate}{enumerate}{9}
\setlist[enumerate,1]{label=(\arabic*)}
\setlist[enumerate,2]{label=(\Roman*)}
\setlist[enumerate,3]{label=(\Alph*)}
\setlist[enumerate,4]{label=(\roman*)}
\setlist[enumerate,5]{label=(\alph*)}
\setlist[enumerate,6]{label=(\arabic*)}
\setlist[enumerate,7]{label=(\Roman*)}
\setlist[enumerate,8]{label=(\Alph*)}
\setlist[enumerate,9]{label=(\roman*)}

\usepackage[capitalise]{cleveref}
\crefname{equation}{}{}
\crefname{prop}{Proposition}{Propositions}
\crefname{enumi}{}{}

\usepackage[utf8]{inputenc}

\newtheorem{thm}{Theorem}[section]
\crefname{thm}{Theorem}{Theorems}
\newtheorem{lem}[thm]{Lemma}
\newtheorem{prop}[thm]{Proposition}
\newtheorem{cor}[thm]{Corollary}
\newtheorem{conj}[thm]{Conjecture}

\newtheorem{claim}[thm]{Claim}
\newtheorem{fact}[thm]{Fact}

\theoremstyle{definition}
\newtheorem{defn}[thm]{Definition}

\numberwithin{equation}{section}

\DeclareMathOperator{\blue}{blue}
\DeclareMathOperator{\red}{red}
\def\al#1{}
	\renewcommand{\al}[1]{\footnote{\textbf{AL: }#1}}  
	
\def\vp#1{}
	\renewcommand{\vp}[1]{\footnote{\textcolor{green!40!black}{\textbf{VP: }#1}}} 
    
\newenvironment{marknew}{\color{red} }{ }
\newcommand{\mn}{\begin{marknew}}
\newcommand{\umn}{\end{marknew}}

\newenvironment{proofclaim}[1][Proof of Claim]{\begin{proof}[#1]}{\end{proof}}

\usepackage[UKenglish]{babel}
\usepackage[UKenglish]{isodate}

\newcommand{\splitatcommas}[1]{%
  \begingroup
  \ifnum\mathcode`,="8000
  \else
    \begingroup\lccode`~=`, \lowercase{\endgroup
      \edef~{\mathchar\the\mathcode`, \penalty0 \noexpand\hspace{0pt plus 1em}}%
    }\mathcode`,="8000
  \fi
  #1%
  \endgroup
}
\newcommand{\listing}[1]{\splitatcommas{#1}} 
\newcommand{\set}[1]{\{\splitatcommas{#1}\}} 

\newcommand{\crel}[1]{%
  \global\setbox1=\hbox{$#1$}%
  \global\dimen1=0.5\wd1
  \mathrel{\hbox to\dimen1{$#1$\hss}}&\mathrel{\mspace{-\thickmuskip}\hbox to\dimen1{}}%
}

\def\COMMENT#1{}
\renewcommand{\COMMENT}[1]{\footnote{#1}}  

\newcommand{\EMAIL}[1]{\textit{{E-mail}}: \texttt{\href{mailto:#1}{#1}}} 

\usepackage{lineno}
\newcommand*\patchAmsMathEnvironmentForLineno[1]{%
  \expandafter\let\csname old#1\expandafter\endcsname\csname #1\endcsname
  \expandafter\let\csname oldend#1\expandafter\endcsname\csname end#1\endcsname
  \renewenvironment{#1}%
     {\linenomath\csname old#1\endcsname}%
     {\csname oldend#1\endcsname\endlinenomath}}%
\newcommand*\patchBothAmsMathEnvironmentsForLineno[1]{%
 \patchAmsMathEnvironmentForLineno{#1}%
  \patchAmsMathEnvironmentForLineno{#1*}}%
\AtBeginDocument{%
\patchBothAmsMathEnvironmentsForLineno{equation}%
\patchBothAmsMathEnvironmentsForLineno{align}%
\patchBothAmsMathEnvironmentsForLineno{flalign}%
\patchBothAmsMathEnvironmentsForLineno{alignat}%
\patchBothAmsMathEnvironmentsForLineno{gather}%
\patchBothAmsMathEnvironmentsForLineno{multline}%
}

\title{The Ramsey number for 4-uniform tight cycles \footnote{
An extended abstract roughly equal to the introduction of this paper has been published as part of the extended abstracts for the EuroComb 2021 conference in a volume of CRM Research Perspectives by Birkhäuser \cite{myabstract}. 
}} 
\author{
Allan Lo 
\thanks{School of Mathematics, University of Birmingham, UK, \EMAIL{s.a.lo@bham.ac.uk}, The research leading to these results was partially supported by EPSRC, grant no. EP/V002279/1 (A.~Lo)} 
\and Vincent Pfenninger \thanks{Institute of Discrete Mathematics, Graz University of Technology, Austria, \EMAIL{pfenninger@math.tugraz.at}}
}
\date{\today}

\begin{document}

\maketitle

\begin{abstract}
A \emph{$k$-uniform tight cycle} is a $k$-graph with a cyclic ordering of its vertices such that its edges are precisely the sets of~$k$ consecutive vertices in that ordering. A \emph{$k$-uniform tight path} is a $k$-graph obtained by deleting a vertex from a $k$-uniform tight cycle. 
We prove that the Ramsey number for the $4$-uniform tight cycle on~$4n$ vertices is $(5 +o(1))n$. This is asymptotically tight. This result also implies that the Ramsey number for the $4$-uniform tight path on~$n$ vertices is $(5/4 + o(1))n$.
\end{abstract}

\section{Introduction}
The \emph{Ramsey number $r(H_1, \dots, H_m)$} for $k$-graphs $H_1, \dots, H_m$ is the smallest integer~$N$ such that any $m$-edge-colouring of the complete $k$-graph~$K_N^{(k)}$ contains a monochromatic copy of~$H_i$ in the $i$-th colour for some $1 \leq i \leq m$.
If $H_1, \dots, H_m$ are all isomorphic to~$H$ then we let $r_m(H) = r(H_1, \dots, H_m)$ and call it the \emph{$m$-colour Ramsey number for~$H$}. We also write~$r(H)$ for~$r_2(H)$ and simply call it the \emph{Ramsey number for~$H$}. 

The Ramsey number for cycles in graphs has been determined in \cite{Bondy1973, Faudree1974, Rosta1973}. In particular, for $n \geq 5$, we have
$$
r(C_n) = \left\{
\begin{array}{ll}
\frac{3}{2}n-1, &{\rm ~if~}n {\rm ~is~even,} \\
2n -1, &{\rm ~if~} n {\rm~is~odd.}
\end{array}
\right.
$$
Note that there is a dependence on the parity of the length of the cycle. For the $m$-colour Ramsey number, Jenssen and Skokan \cite{Jenssen2021} proved that for $m \geq 2$ and any large enough odd integer~$n$ we have $r_m(C_n) = 2^{m-1}(n-1) +1$.

Some Ramsey numbers for $k$-graphs have also been studied.
A $k$-uniform tight cycle~$C_n^{(k)}$ is a $k$-graph on~$n$ vertices with a cyclic ordering of its vertices such that its edges are the sets of~$k$ consecutive vertices.
The Ramsey number of the $3$-uniform tight cycle on~$n$ vertices~$C_{n}^{(3)}$ was determined asymptotically by Haxell, {\L}uczak, Peng, R\"odl, Ruci\'nski and Skokan, see \cite{Haxell2007, Haxell2009}. They showed that $r(C_{3n}^{(3)}) = (1+o(1))4n$ and $r(C_{3n+i}^{(3)}) = (1+o(1))6n$ for $i \in \{1,2\}$.

We define the \emph{$k$-uniform tight path on~$n$ vertices~$P_n^{(k)}$} to be the $k$-graph obtained from $C_{n+1}^{(k)}$ by deleting a vertex. Using the bound on the Tur\'{a}n number for tight paths that was recently shown by F\"{u}redi, Jiang, Kostochka, Mubayi and Verstra{\"e}te \cite{Furedi2020}, one can deduce that $r(P_n^{(k)}) \leq k(n-k+1)$ for any even $k \geq 2$.

The Ramsey number for loose cycles have also been studied. 
We denote by~$LC_n^{(k)}$, where $n = \ell(k-1)$, the $k$-uniform loose cycle on~$n$ vertices, that is the $k$-graph with vertex set $\{v_1, \dots, v_n\}$ and edges $e_i = \{v_{1 + i(k-1)}, \dots, v_{k + i(k-1)}\}$ for $0 \leq i \leq \ell -1$, where indices are taken modulo~$n$.
Gy\'{a}rf\'{a}s, S\'{a}rk\"{o}zy and Szemer\'{e}di \cite{Gyarfas2008} showed that for $k \geq 3$ and~$n$ divisible by~$k-1$ that
\[
r(LC_n^{(k)}) = (1 + o(1))\frac{2k-1}{2k-2}n.
\]
Recently, the exact values of Ramsey numbers for loose cycles have been determined for various cases, see \cite{Shahsiah2018} for more details.

Another problem of interest in this area is determining the Ramsey number of a complete graph and a cycle.
For graphs, Keevash, Long and Skokan \cite{Keevash2021} showed that there exists an absolute constant $C \geq 1$ such that 
\[
r(C_\ell, K_n) = (\ell -1)(n-1)+1 {\rm ~provided~} \ell \geq  \frac{C\log n}{\log \log n}.
\]
Analogous problems for hypergraphs have also been considered. See \cite{Meroueh2019,Mubayi2016,Nie2021} for the analogous problem with loose, tight and Berge cycles, respectively.

In this paper, we will consider the Ramsey number for tight cycles.
We determine the Ramsey number for the $4$-uniform tight cycle on~$n$ vertices~$C_n^{(4)}$ asymptotically in the case where~$n$ is divisible by~$4$.

\begin{thm}
\label{thm:main}
Let $\varepsilon > 0$. For~$n$ large enough we have $r(C_{4n}^{(4)}) \leq (5 + \varepsilon)n$.
\end{thm}

It is easy to see that this is asymptotically tight.

\begin{prop}
\label{prop:tight}
For $n, k \geq 2$, we have that $r(C_{kn}^{(k)}) \geq (k+1)n - 1$.
\end{prop}
\begin{proof}
Let $N = (k+1)n-2$. We show that there exists a red-blue edge-colouring of~$K_N^{(k)}$ that does not contain a monochromatic copy of $C_{kn}^{(k)}$. We partition the vertex set of~$K_N^{(k)}$ into two sets~$X$ and~$Y$ of sizes~$n-1$ and~$kn-1$, respectively. We colour every edge that intersects the set~$X$ red and every other edge blue. It is easy to see that this red-blue edge-colouring of~$K_N^{(k)}$ does not even contain a monochromatic matching of size~$n$ and thus also cannot contain a monochromatic copy of $C_{kn}^{(k)}$. Indeed, there is no red matching of size~$n$ since every red edge must intersect~$X$ and $|X| = n-1$. Moreover, there is no blue matching of size~$n$ since all blue edges are entirely contained in~$Y$ and $|Y| = kn-1$.
\end{proof}

It is clear that the proof of \cref{prop:tight} also shows that $r(P_{4n+i}^{(4)}) \geq 5n - 1$ for $0 \leq i \leq 3$. Since $C_{4(n+1)}^{(4)}$ contains $P_{4n+i}^{(4)}$ for each $0 \leq i \leq 3$, \cref{thm:main} also determines the Ramsey number for the $4$-uniform tight path asymptotically. 

\begin{cor}
\label{cor:tight_path}
We have $r(P_n^{(4)}) = (5/4 + o(1))n$.
\end{cor}

\subsection{Sketch of the proof of Theorem~\ref{thm:main}} 
\label{subsection:introduction_proof_sketch}

We now sketch the proof of Theorem~\ref{thm:main}.
We use a hypergraph version of the connected matching method of {\L}uczak \cite{Luczak1999} as follows.
We consider a red-blue edge-colouring of~$K_N^{(4)}$ for $N = (5/4 + \varepsilon)n$. We begin by applying the Hypergraph Regularity Lemma. More precisely, we use the Regular Slice Lemma of Allen, B\"ottcher, Cooley and Mycroft~\cite{Allen2017}. This gives us a reduced graph $\mathcal{R}$, which is a red-blue edge-coloured almost complete $4$-graph on $(5/4 +\varepsilon)n'$ vertices.
To prove Theorem~\ref{thm:main}, it now suffices to find a \emph{monochromatic tightly connected matching} of size~$n'/4$ in~$\mathcal{R}$. A monochromatic tightly connected matching is a monochromatic matching~$M$ such that for any two edges $f, f' \in M$, there exists a tight walk\footnote{A \emph{tight walk} in a $k$-graph is a sequence of edges $e_1, \dots, e_t$ such that $|e_i \cap e_{i+1}| = k-1$ for all $1 \leq i \leq t-1$.} in~$\mathcal{R}$ of the same colour as~$M$ connecting~$f$ and~$f'$.
This reduction of our problem to finding a monochromatic tightly connected matching in the reduced graph is formalised in \cref{cor:matchings_to_cycles}, which was proved by the authors in an earlier paper \cite{Lo2020}. There will be no need for us to introduce the Hypergraph Regularity Lemma in this paper as the application of it is hidden within \cref{cor:matchings_to_cycles}.

Let~$\gamma$ be a constant such that $0 < \gamma \ll \varepsilon$ and let~$M$ be a maximal monochromatic tightly connected matching in~$\mathcal{R}$. Suppose that~$M$ has size less than~$n'/4$ and is red. We show that we can find a monochromatic tightly connected matching of size at least $\s{M} + \gamma n'$. By iterating this we get our desired result.
We actually find a \emph{fractional matching} instead. Note that by taking a blow-up of~$\mathcal{R}$ we can then convert it back into an integral matching.
For simplicity, let us further assume that $\mathcal{R}$ has only one red and one blue tight component (see \cref{section:Preliminaries} for the definition). Then any monochromatic matching is tightly connected. Consider an edge $f \in M$ and a vertex~$w$ not covered by~$M$. Observe that if all the edges in $\mathcal{R}[f \cup \{w\}]$ are red, then we get a larger red fractional matching (by giving weight~$1/4$ to each of the five edges in $\mathcal{R}[f \cup \{w\}]$). Thus for almost all the edges $f \in M$ there is a blue edge~$f'$ such that $|f \cap f'| =3$. This gives us a blue matching~$M'$ of almost the same size as~$M$. Note that the set of leftover vertices $W = V(\mathcal{R}) \setminus V(M \cup M')$ has size at least $\varepsilon n'$. By the maximality of~$M$, any edge in $\mathcal{R}[W]$ must be blue. So we can extend~$M'$ by adding a matching in~$W$ to get the desired matching of size at least $\s{M} + \gamma n'$.

However, $\mathcal{R}$ may contain many monochromatic tight components (instead of just two). Hence we need to choose monochromatic tight components carefully.
To do this we use a novel auxiliary graph called the blueprint which the authors introduced in~\cite{Lo2020}. The blueprint is a graph with the key property that monochromatic connected components in it correspond to monochromatic tight components in the $4$-graph we are considering.
Since the blueprint is red-blue edge-coloured and almost complete, it contains an almost-spanning monochromatic tree. Using the key property of blueprints this shows that $\mathcal{R}$ contains a large monochromatic tight component.

\subsection{Organisation of the Paper}
In \cref{section:Preliminaries}, we introduce basic notation and definitions. In \cref{section:blow-ups}, we define blow-ups and prove some basic propositions about them.
In \cref{section:blueprints}, we define blueprints, state a result about their existence and prove some basic results about how they interact with blow-ups.
In \cref{section:good_fractional_matchings}, we prove that an almost complete $2$-edge-coloured $4$-graph contains a monochromatic tightly connected fractional matching with large weight.
In \cref{section:proof_of_thm}, we show how to use this to prove \cref{thm:main}.

\section{Preliminaries}
\label{section:Preliminaries}
For a positive integer~$k$, we let $[k] = \{1, \dots, k\}$.
We often write $x_1\dots x_j$ for the set $\{x_1, \dots, x_j\}$. 
For a set~$S$ and a non-negative integer~$k$ we denote by $\binom{S}{k}$ the set of all subsets of~$S$ of size~$k$.
In this paper, we omit floors and ceilings whenever doing so does not affect the argument. We say that a statement holds for $0 \ll a \ll b < 1$ if there exists a non-decreasing function $f \colon (0,1] \rightarrow (0,1]$ such that the statement holds for all choices of~$a$ and~$b$ such that $a \leq f(b)$. Hierarchies with more variables are defined similarly. If a reciprocal~$1/n$ appears in such a hierarchy, we implicitly assume that~$n$ is a positive integer.

A \emph{$k$-graph}~$H$ is a pair of sets $(V(H), E(H))$ such that $E(H) \subseteq \binom{V(H)}{k}$. We let $v(H) = \s{V(H)}$ be the number of vertices of $H$.
We abuse notation by identifying the $k$-graph~$H$ with its edge set~$E(H)$.
If~$H$ is a $k$-graph and $\mathcal{P} = \{V_1, \dots V_s\}$ a partition of~$V(H)$, then we call an edge $e \in H$ \emph{$\mathcal{P}$-partite} if $\s{V_i \cap e} \leq 1$ for all $i \in [s]$. If all edges of~$H$ are $\mathcal{P}$-partite, then we call~$H$ \emph{$\mathcal{P}$-partite}. We call $H$ $s$-partite if $H$ is $\mathcal{P}$-partite for some partition $\mathcal{P}$ of $V(H)$ into $s$ sets.
For vertex-disjoint $k$-graphs~$H_1$ and~$H_2$, we define the $k$-graph $H_1 \cup H_2 = (V(H_1) \cup V(H_2), E(H_1) \cup E(H_2))$. 

For a $k$-graph~$H$, a pair of edges $f, f' \in E(H)$ is called \emph{tightly connected in~$H$} if there exists a sequence of edges $e_1, \dots, e_t \in F$ such that $e_1 = f$, $e_t = f'$ and $\s{e_i \cap e_{i+1}} = k-1$ for every $i \in [t-1]$.
A subgraph~$H'$ of~$H$ is called \emph{tightly connected} if each pair of edges in~$H'$ is tightly connected in~$H'$.
A \emph{tight component} in a $k$-graph~$H$ is a maximal tightly connected subgraph. Note that a tight component is a subgraph rather than just a set of vertices (unlike the usual definition for a component in a graph). In a $2$-graph~$G$, we simply call a tight component a \emph{component} and a \emph{spanning component} is one that covers all the vertices of~$G$. A \emph{tightly connected matching} in a $k$-graph~$H$ is a matching contained in a tight component of~$H$.

A $2$-edge-coloured $k$-graph is a $k$-graph together with a colouring of its edges with the colours red and blue.
For a $2$-edge-coloured $k$-graph~$H$ we denote by~$H^{\red}$ and~$H^{\blue}$ the $k$-graph on~$V(H)$ induced by the red and the blue edges of~$H$, respectively. 
A subgraph of a $2$-edge-coloured $k$-graph is called \emph{monochromatic} if all its edges have the same colour. 

A \emph{fractional matching} in a $k$-graph~$H$ is a function $\varphi \colon E(H) \rightarrow [0,1]$ such that for every $v \in V(H)$, $\sum_{e \in E(H): v \in e} \varphi(e) \leq 1$. For each $e \in E(H)$ we call~$\varphi(e)$ the \emph{weight of~$e$}. The \emph{weight of~$\varphi$} is $\sum_{e \in E(H)} \varphi(e)$. For a positive integer~$r$, a \emph{$1/r$-fractional matching~$\varphi$} in a $k$-graph~$H$ is a fractional matching such that each edge has weight in $\{0, \frac{1}{r}, \frac{2}{r}, \dots, \frac{r-1}{r}, 1\}$, that is $\{\varphi(e) \colon e \in E(H)\} \subseteq \{0, \frac{1}{r}, \frac{2}{r}, \dots, \frac{r-1}{r}, 1\}$.
For vertex-disjoint $k$-graphs~$H_1$ and~$H_2$ and fractional matchings~$\varphi_1$ and~$\varphi_2$ in~$H_1$ and~$H_2$, respectively, we define the fractional matching $\varphi_1 + \varphi_2 \colon E(H_1 \cup H_2) \rightarrow [0,1]$ in $H_1 \cup H_2$ by setting $(\varphi_1 +\varphi_2)(e) = \varphi_1(e)$ if $e \in H_1$ and $(\varphi_1 +\varphi_2)(e) = \varphi_2(e)$ if $e \in H_2$. For a $k$-graph~$H$, a subgraph~$H'$ of~$H$ and a fractional matching~$\varphi$ in~$H'$, we define the \emph{completion of~$\varphi$ with respect to~$H$}, denoted~$\varphi^H$, to be the fractional matching $\varphi^H \colon E(H) \rightarrow [0,1]$ such that $\varphi^H(e) = \varphi(e)$ if $e \in H'$ and $\varphi(e) = 0$ otherwise. 
For a matching~$M$ in a $k$-graph~$H$, we define the \emph{fractional matching induced by the matching~$M$} to be the fractional matching $\varphi \colon E(M) \rightarrow [0,1]$ with $\varphi(e) = 1$ for all $e \in M$. 
A \emph{tightly connected fractional matching} in a $k$-graph~$H$ is a fractional matching~$\varphi \colon E(H') \rightarrow [0,1]$ where $H'$ is a tight component of~$H$.

A \emph{red tight component}, a \emph{red tightly connected matching} and a \emph{red tightly connected fractional matching} in a $2$-edge-coloured $k$-graph~$H$ are a tight component, a tightly connected matching and a tightly connected fractional matching, respectively, in~$H^{\red}$. We define these terms analogously for blue. A \emph{monochromatic tight component} in~$H$ is a red or a blue tight component in~$H$ and similarly for the other terms.

\section{Blow-ups}
\label{section:blow-ups}

We will later need blow-ups to convert fractional matchings to integral ones. So we define blow-ups here and show some basic facts.

\begin{defn}
Given a $k$-graph~$H$ we say that~$H_*$ is a blow-up of~$H$ if there exists a partition $\mathcal{P} = \{V_x \colon x \in V(H)\}$ of~$V(H_*)$ such that $H_* = \bigcup_{x_1 \dots x_k \in H} K_{V_{x_1}, \dots, V_{x_k}}$,
where $K_{V_{x_1}, \dots, V_{x_k}}$ is the complete $k$-partite $k$-graph with vertex classes $V_{x_1}, \dots, V_{x_k}$. 
Moreover, if~$H$ is $2$-edge-coloured, then we have
\[
\text{$H_*^{\red} = \bigcup_{x_1 \dots x_k \in H^{\red}} K_{V_{x_1}, \dots, V_{x_k}}$ and $H_*^{\blue} = \bigcup_{x_1 \dots x_k \in H^{\blue}} K_{V_{x_1}, \dots, V_{x_k}}$}.
\]
If $\s{V_x} = r$ for all $x \in V(H)$, then we call~$H_*$ an \emph{$r$-blow-up} of~$H$.
If $e_* = y_1 \dots y_k \in H_*$, then we let $f_{e_*} = x_1 \dots x_k$ be the unique edge in~$H$ such that $y_i \in V_{x_i}$ for all $i \in [k]$. 
\end{defn}

A $k$-graph~$H$ on~$n$ vertices is called \emph{$(\mu, \alpha)$-dense} if, for each $i \in [k-1]$, we have $d_H(S) \geq \mu \binom{n}{k-i}$ for all but at most $\alpha \binom{n}{i}$ sets $S \in \binom{V(H)}{i}$ and $d_H(S) = 0$ for all other sets $S \in \binom{V(H)}{i}$.
The following proposition shows that the $r$-blow-up of a $(1-\eps, \alpha)$-dense $k$-graph is $(1-2\eps, 2\alpha)$-dense.

\begin{prop}
\label{prop:H_*_is_dense}
Let $1/n \ll \eps, \alpha, 1/r, 1/k$. Let~$H$ be a $(1-\eps, \alpha)$-dense $k$-graph on~$n$ vertices and let~$H_*$ be an $r$-blow-up of~$H$. Then~$H_*$ is $(1- 2\eps, 2 \alpha)$-dense. 
\end{prop}
\begin{proof}
Let $\mathcal{P} = \{V_x \colon x \in V(H)\}$ be the partition of~$V(H_*)$.
Let $i \in [k-1]$ and $S \in \binom{V(H_*)}{i}$ be such that~$S$ is $\mathcal{P}$-partite. Let $S' = x_1 \dots x_i \in \binom{V(H)}{i}$ be such that $S \in K_{V_{x_1}, \dots, V_{x_i}}$. Suppose $d_{H}(S') \geq (1-\eps)\binom{n}{k-i}$. Then $d_{H_*}(S) \geq (1-\eps) \binom{n}{k-i}r^{k-i} \geq (1-2\eps)\binom{nr}{k-i}$, since $\binom{nr}{k-i} = (1+o(1))\binom{n}{k-i}r^{k-i}$ as $n \rightarrow \infty$. The number of sets $S \in \binom{V(H_*)}{i}$ for which this is true is at least $r^i(1-\alpha)\binom{n}{i} \geq (1-2\alpha)\binom{nr}{i}$.
For all other sets $S \in \binom{V(H_*)}{i}$, we have $d_{H_*}(S) = 0$. Hence~$H_*$ is $(1-2\eps, 2\alpha)$-dense.
\end{proof}

The following proposition shows how to turn a matching in a blow-up of a $k$-graph~$H$ into a fractional matching in~$H$ and vice versa.

\begin{prop}
\label{prop:blow-up_to_fractional}
Let $1/N \ll 1/r$ and $k \geq 2$. Let~$H$ be an edge-coloured $k$-graph on~$N$ vertices and let~$H_*$ be an $r$-blow-up of~$H$. Then~$H_*$ contains a monochromatic tightly connected matching~$M$ of size~$m$ if and only if~$H$ contains a monochromatic tightly connected $1/r$-fractional matching of weight~$m/r$ of the same colour.
\end{prop}
\begin{proof}
Let $\mathcal{P} = \{V_x \colon x \in V(H) \}$ be the partition of~$V(H_*)$.
Let~$F_*$ be the monochromatic tight component of~$H_*$ that contains~$M$. There exists a monochromatic tight component~$F$ of~$H$ such that 
\[
F_* = \bigcup_{x_1 \dots x_k \in F} K_{V_{x_1}, \dots, V_{x_k}}.
\]
We define the fractional matching $\varphi \colon F \rightarrow [0,1]$ as follows. For each edge $e = x_1 \dots x_k \in F$, we set
\[
\varphi(e) = \frac{\s{M \cap K_{V_{x_1}, \dots, V_{x_k}}}}{r}.
\]
For each $x \in V(H)$, 
\[
\sum_{e \in H \colon x \in e} \varphi(e) = \frac{1}{r} \s{\{f \in M \colon f \cap V_x \neq \varnothing\}} \leq 1,
\]
since~$M$ is a matching and $\s{V_x} = r$. Hence~$\varphi$ is a monochromatic tightly connected $1/r$-fractional matching.
We conclude by noting that~$\varphi$ has weight~$m/r$. 

The other direction is proved similarly.
\end{proof}

\section{Blueprints}
\label{section:blueprints}
We use the notion of blueprint introduced in \cite{Lo2020}, which allows us to track monochromatic tight components.

\begin{defn}
\label{defn:blueprint}
Let~$\eps >0$, $k \geq 3$ and let~$H$ be a $2$-edge-coloured $k$-graph on~$n$ vertices. We say that a $2$-edge-coloured $(k-2)$-graph~$G$ with $V(G) \subseteq V(H)$ is an \emph{$\eps$-blueprint for~$H$}, if
\begin{enumerate}[label = {$(\text{BP}\arabic*)$}, leftmargin= \widthof{BP1000}]
    \item \label{BP1} for every edge~$e \in G$, there exists a monochromatic tight component~$H(e)$ in~$H$ such that~$H(e)$ has the same colour as~$e$ and $d_{\partial H(e)}(e) \geq (1-\eps)n$ and
    \item \label{BP2} for $e, e' \in G$ of the same colour with $\s{e \cap e'} = k-3$, we have $H(e) = H(e')$.
\end{enumerate}
We say that~$e$ induces~$H(e)$ and write~$R(e)$ or~$B(e)$ instead of~$H(e)$ if~$e$ is red or blue, respectively.
We simply say that~$G$ is a \emph{blueprint}, when~$H$ is clear from context and there exists $\eps > 0$ such that~$G$ is an $\eps$-blueprint for~$H$. For $S \in \binom{V(H)}{k-3}$, all the red (blue) edges of a blueprint containing~$S$ induce the same red (blue) tight component, so we call that component the red (blue) tight component induced by~$S$. Note that any subgraph of a blueprint is also a blueprint.

For a monochromatic tight component~$K$ of~$H$, we let $K_G^{k-2} = \{e \in G \colon H(e) = K\}$ and we drop the subscript if~$G$ is clear from context. In other words,~$K_G^{k-2}$ is the subgraph of~$G$ whose edges induce~$K$.
\end{defn}

The following proposition shows blueprints work well together with blow-ups.

\begin{prop}
\label{prop:G_*_is_blueprint}
Let $1/n \ll \eps \ll 1/k \leq 1/4$ and let $r \geq 2$ be an integer.
Let~$H$ be a $2$-edge-coloured $k$-graph on~$n$ vertices, let~$G$ be an $\eps$-blueprint for~$H$, and~$H_*$ an $r$-blow-up of~$H$ with vertex partition $\mathcal{P} = \{V_x \colon x \in V(H)\}$.
Let $G_*$ be the $r$-blow-up of~$G$ with vertex partition $\mathcal{P}' = \{V_x \colon x \in V(G)\}$. 
Then~$G_*$ is an $\eps$-blueprint for~$H_*$. Moreover, for $e_* = y_1 \dots y_{k-2} \in G_*$ and $f_{e_*} = x_1 \dots x_{k-2} \in G$ where $y_i \in V_{x_i}$ for all $i \in [k-2]$, we have $H_*(e_*) = \bigcup_{z_1 \dots z_k \in H(f_{e_*})} K_{V_{z_1}, \dots, V_{z_k}}$, that is $H_*(e_*)$ is the $r$-blow-up of $H(f_{e_*})$ in $H_*$.
\end{prop}
\begin{proof}
For $e_* = y_1 \dots y_{k-2} \in G_*$, we let $H_*(e_*)$ be the blow-up of $H(f_{e_*})$ with respect to $\mathcal{P}$, that is $H_*(e_*) = \bigcup_{z_1 \dots z_k \in H(f_{e_*})} K_{V_{z_1}, \dots, V_{z_k}}$.
Since $H(f_{e_*})$ is a monochromatic tight component in $H$, $H_*(e_*)$ is indeed a monochromatic tight component in $H_*$ as required. Moreover, since $f_{e_*}$ has the same colour as $e_*$, $H(e_*)$ has the same colour as $e_*$.

Let $e_* \in G_*$. We show that $d_{\partial H_*(e_*)}(e_*) \geq (1-\eps) n r$.
Since~$H_*(e_*)$ is the blow-up of~$H(f_{e_*})$ with respect to $\mathcal{P}$, $\partial H_*(e_*)$ is the blow-up of $\partial H(f_{e_*})$ with respect to $\mathcal{P}$. It follows that $d_{\partial H_*(e_*)}(e_*) = r d_{\partial H(f_{e_*})}(f_{e_*}) \geq (1 - \eps) n r$.

Now let $e_*, e_*' \in G_*$ of the same colour with $\s{e_* \cap e_*'} = k-3$. We show that $H_*(e_*) = H_*(e_*')$. 
We have $\s{f_{e_*} \cap f_{e_*'}} = k-3$ and~$f_{e_*}$ and~$f_{e_*'}$ have the same colour. Thus since~$G$ is a blueprint for~$H$, we have $H(f_{e_*}) = H(f_{e_*'})$. Thus, by definition, $H_*(e_*) = H_*(e_*')$. 
\end{proof}

The blueprint of a $2$-edge-coloured $4$-graph is a $2$-edge-coloured graph. 
We use the following proposition to show that the blow-up of such a blueprint retains the properties of having large minimum degree and of having a spanning red component. 

\begin{prop}
\label{prop:G_*_min_degree}
Let $1/n \ll \beta, r$. Let~$G$ be a $2$-edge-coloured $2$-graph with $\delta(G) \geq (1-\beta)n$ and let~$G_*$ be the $r$-blow-up of~$G$ with vertex partition $\mathcal{P} = \{V_x \colon x \in V(G)\}$. Then $\delta(G_*) \geq (1-\beta) nr$. Further, if~$G$ contains a spanning red component, then~$G_*$ contains a spanning red component. The same statement holds with the colours reversed. 
\end{prop}
\begin{proof}
Let $v \in V(G_*)$. There exists $x \in V(G)$ such that $v \in V_x$. We have $N_{G_*}(v) = \bigcup_{y \in N_G(x)} V_y$ and thus $d_{G_*}(v) \geq \delta(G) r \geq (1-\beta)nr$. Now assume that~$G$ contains a spanning red component. We show that~$G_*$ contains a spanning red component. Let $u, v \in V(G_*)$. There exist $x, y \in V(G)$ such that $u \in V_x$ and $v \in V_y$. Let $z \in V(G) \setminus \{x, y\}$. Since~$G$ contains a spanning red component, there exist walks $x x_1 \dots x_{k} z$ and $y y_1 \dots y_{\ell} z$ in~$G^{\red}$. Choose vertices $u_{i} \in V_{x_i}$, $v_j \in V_{y_i}$, and $v_z \in V_z$ for $i \in [k]$ and $j \in [\ell]$. Note that $u u_1 \dots u_k z v_\ell \dots v_1 v$ is a walk in~$G_*^{\red}$. Hence~$G_*$ contains a spanning red component.
\end{proof}

We use the following two results to show that a $2$-edge-coloured almost complete $4$-graph has a blueprint with large minimum degree that contains a spanning monochromatic component.

\begin{lem}[Lemma 23 in \cite{Lo2020}]
\label{lem:generalblueprint}
Let $1/n \ll \eps \leq \alpha \ll 1/k \leq 1/3$. Let~$H$ be a $2$-edge-coloured $(1-\eps, \alpha)$-dense $k$-graph on~$n$ vertices.
Then there exists a $3\sqrt{\eps}$-blueprint~$G_*$ for~$H$ with $V(G_*) = V(H)$ and $\s{G_*} \geq (1-\alpha -24k\sqrt{\eps})\binom{n}{k-2}$.
\end{lem}

\begin{cor}[Corollary 26 in \cite{Lo2020}]
\label{cor:Ecomp}
Let $1/n \ll \eps \leq 1/324$. Let~$F$ be a $2$-edge-coloured $2$-graph with $\s{ V(F) } \, \leq n$ and $\s{ E(F) } \geq (1-\eps)\binom{n}{2}$. Then there exists a subgraph~$F'$ of~$F$ of order at least $(1-3\sqrt{\eps})n$ that contains a spanning monochromatic component and $\delta(F') \geq (1-6\sqrt{\eps})n$.
\end{cor}

\section{Finding monochromatic tightly connected matchings}
\label{section:good_fractional_matchings}

Our goal in this section is to prove the following lemma which is the main ingredient in the proof of \cref{thm:main}.

\begin{lem}
\label{lem:main_matchings_lemma}
Let $1/n \ll \eps \ll c \ll \eta$. Let $N =  (5/4 + 3 \eta)n $. Let~$H$ be a $2$-edge-coloured $(1-\eps, \eps)$-dense $4$-graph on~$N$ vertices. Then there exists a monochromatic tightly connected fractional matching in~$H$ with weight at least~$n/4$ and all weights at least~$c$.
\end{lem}

By using \cref{prop:blow-up_to_fractional}, proving \cref{lem:main_matchings_lemma} is reduced to showing that, for some $\eps \ll 1/s \ll \eta$, an $s$-blow-up~$H_*$ of~$H$ contains a monochromatic tightly connected matching of size at least~$v(H_*)/5$. We will prove this as follows. First we find a monochromatic tightly connected matching in~$H$ of size $\delta n$ for some $1/s \ll \delta \ll \eta$. We then iteratively take blow-ups of~$H$ that contain monochromatic tightly connected matchings that cover a larger and larger proportion of the vertices. We prove this by showing that as long as our current matching~$M$ in a blow-up~$H_*$ of~$H$ is not yet large enough, we can find a fractional monochromatic matching of weight $\s{M} + \gamma v(H_*)$ (where $1/s \ll \gamma \ll \delta$). This is the main ingredient in the proof of \cref{lem:main_matchings_lemma} and is formalised in \cref{lem:increase_matching}. We then convert this fractional matching into an integral matching by taking another blow-up (see \cref{prop:make_integral}).

Since~$H$ is only almost complete, its blueprint will also only be almost complete. To overcome some difficulties arising from this, we mostly work with edges of~$H$ that work well with respect to its blueprint. We call these edges \emph{good edges} and define them as follows.

\begin{defn}[Good edges, good sets of edges, good fractional matchings,~$J^+$]
Let~$H$ be a $2$-edge-coloured $4$-graph and~$G$ a blueprint for~$H$. We call an edge $f \in H$ \emph{good for~$(H,G)$} if
\begin{enumerate}[label = \textup{(G\arabic*)}, leftmargin= \widthof{G1000}]
    \item $f \subseteq V(G)$, \label{G1}
    \item $G[f] \cong K_4$ and \label{G2}
    \item there exists $z \in f$ such that $xyz \in \partial H(xy)$ for every $xy \in \binom{f \setminus \{z\}}{2}$. \label{G3}
\end{enumerate}
We call a set of edges $F \subseteq H$ \emph{good for~$(H,G)$}, if every edge $f \in F$ is good for~$(H,G)$. For a subgraph~$J$ of~$H$, a fractional matching $\varphi \colon E(J) \rightarrow [0,1]$ is called \emph{good for~$(H,G)$} if $\{e \in E(J) \colon \varphi(e) > 0\}$ is good for~$(H,G)$. If~$H$ and~$G$ are clear from context, then we simply call such edges, sets of edges and fractional matchings \emph{good}.
For a subgraph~$J$ of~$H$, we let~$J^+$ be the set of edges $f \in J$ such that~$f$ is good.
\end{defn}

Intuitively, by using only good edges we can ignore some of the problems that arise from the fact that~$H$ and~$G$ are only almost complete. The purpose of \cref{G3} is to allow us to deduce, in some situations, that the edge~$f$ is in one of the monochromatic tight components induced by the edges of~$G[f]$. For example, if $f = x_1x_2x_3x_4$ is a blue edge in~$H$ and $x_1x_2, x_3x_4 \in G^{\blue}$, then \cref{G3} implies that $f \in B(x_1x_2) \cup B(x_3x_4)$.

The following lemma is the main ingredient in the proof of \cref{lem:main_matchings_lemma}. It says that if we have a monochromatic matching that is not large enough, then we can find a larger one.

\begin{lem} 
\label{lem:increase_matching}
Let $r = \binom{9}{4}!$.
Let $1/n \ll \eps \ll \gamma \ll \delta \ll \eta \ll 1$. Let $N =  (5/4 + 3 \eta)n $. Let~$H$ be a $2$-edge-coloured $(1-\eps, \eps)$-dense $4$-graph on~$N$ vertices that does not contain a monochromatic tightly connected matching of size at least~$n/4$. Let~$G$ be an $\eps$-blueprint for~$H$ with $\delta(G) \geq (1- \eps)N$. Suppose~$H$ contains a red tight component~$R$ satisfying $H(e) = R$ for every $e \in G^{\red}$. Let~$M$ be a good matching in~$H$ with $3\delta n \leq \s{M} < n/4$ such that one of the following holds.
\begin{enumerate}[label = \textup{(H\arabic*)}, leftmargin= \widthof{H1000}]
    \item $M \subseteq R$ \label{H1} or
    \item~$M$ is contained in a blue tight component~$B$ of~$H$. \label{H2}
\end{enumerate}
Then~$H$ contains a good $1/r$-fractional matching in~$R$ or in a blue tight component of~$H$ of weight at least $\s{M} + \gamma n$. The same statement holds with colours reversed.
\end{lem}

\subsection{\texorpdfstring{Proof of \cref{lem:main_matchings_lemma} assuming \cref{lem:increase_matching}}{Proof of Lemma 5.1 assuming Lemma 5.3}}

Before proving \cref{lem:increase_matching}, we show how to prove \cref{lem:main_matchings_lemma} using \cref{lem:increase_matching}. To do this we need a few other small results.

The following proposition shows that we can turn a good $1/r$-fractional matching into a good integral matching by taking an $r$-blow-up.

\begin{prop}
\label{prop:make_integral}
Let $n, r \geq 2$ be integers and $\mu \geq 0$. Let~$H$ be a $2$-edge-coloured $4$-graph on~$n$ vertices. Let~$H_*$ be an $r$-blow-up of~$H$ with vertex partition $\mathcal{P} = \{V_x \colon x \in V(H) \}$. Let~$G$ be a blueprint for~$H$ and let~$G_*$ be the corresponding blueprint for~$H_*$ as defined in \cref{prop:G_*_is_blueprint}. Let~$F$ be a monochromatic tight component of~$H$ and let $F_* = \bigcup_{x_1 \dots x_4 \in F} K_{V_{x_1}, \dots, V_{x_4}}$ be the corresponding monochromatic tight component of~$H_*$. Let~$\varphi$ be a $1/r$-fractional matching in~$F$ with weight~$\mu$ that is good for $(H, G)$. Then there exists an (integral) matching in~$F_*$ of size $\mu r$ that is good for~$(H_*,G_*)$. 
\end{prop}
\begin{proof}
For each vertex $x \in V(H)$ and each edge $e \in F$ containing~$x$, choose disjoint sets $U_{x,e} \subseteq V_x$ such that $\s{U_{x, e}} = r \varphi(e)$. This is possible since~$\varphi$ is a $1/r$-fractional matching and $\s{V_x} = r$ for each $x \in V(H)$. For each edge $e = x_1 \dots x_4 \in F$, let~$M_e$ be a perfect matching in $K_{U_{x_1,e}, \dots, U_{x_4,e}}$. Clearly, the~$M_e$ are disjoint. Let $M = \bigcup_{e \in F} M_e$. Note that $\s{M} = \sum_{e \in F} r \varphi(e) = \mu r$. It is easy to see that since~$\varphi$ is a fractional matching in~$F$ that is good for $(H, G)$,~$M$ is a matching in~$F_*$ that is good for $(H_*, G_*)$.
\end{proof}

The following proposition shows that (in a strong sense) most edges are good in our usual setting of having a $(1-\eps, \eps)$-dense $2$-edge-coloured $4$-graph~$H$ and a blueprint~$G$ for it with large minimum degree.

\begin{prop}
\label{prop:greedy_matching}
Let $1/N \ll \eps \ll \gamma$. Let~$H$ be a $(1-\eps, \eps)$-dense $2$-edge-coloured $4$-graph on~$N$ vertices, let~$G$ be an $\eps$-blueprint for~$H$ with $\delta(G) \geq (1-\eps)N$ and $W \subseteq V(G)$ a set of size at least $\gamma N$. Then $\delta_1(H^+[W]) \geq \binom{\s{W}}{3} - 2 \eps N^3$. Moreover,~$H^+[W]$ contains a matching of size at least $\frac{\s{W}}{4} - \gamma N$.
\end{prop}
\begin{proof}
Fix $v \in V(G)$. Choose vertices
\begin{align*}
    z_1 &\in N_G(v), \\
    z_2 &\in N_G(v) \cap N_G(z_1) \cap N_{\partial H(vz_1)}(vz_1) \text{ and }\\ 
    z_3 &\in N_G(v) \cap N_G(z_1) \cap N_G(z_2) \cap \bigcap_{xy \in \binom{v z_1 z_2}{2}} N_{\partial H(xy)}(xy) \cap N_H(vz_1z_2).
\end{align*}
Note that~$vz_1z_2z_3$ is good.
Since~$G$ is an $\eps$-blueprint for~$H$ with $\delta(G) \geq (1 - \eps)N$ and~$H$ is $(1-\eps, \eps)$-dense, the number of choices for~$z_1$,~$z_2$ and~$z_3$ are at least~$(1-\eps)N$, $(1 - 3\eps)N$ and $(1 - 7\eps)N$, respectively. 
Hence 
\[
\delta_1(H^+[V(G)]) \geq \frac{(1-11\eps)N^3}{3!} \geq \binom{N}{3} - 2 \eps N^3.
\]
It follows that $\delta_1(H^+[W]) \geq \binom{\s{W}}{3} - 2 \eps N^3$. By a greedy argument,~$H^+[W]$ contains a matching of size at least $\frac{\s{W}}{4} - \gamma N$.
\end{proof}

The following proposition shows that in our usual setting the following holds.
Given two sets of vertices~$T_1$ and~$T_2$ (with $\s{T_i} \in \{2,3,4\}$ and satisfying some simple conditions) there exist vertices $z_1, z_2, z_3$ such that all the edges in $H[T_i \cup z_1z_2z_3]$ are good. It is also shown that we can choose these vertices $z_1, z_2$ and~$z_3$ in any not too small set of vertices. We use this to find tight connections of good edges between small sets of vertices.

\begin{prop}
\label{prop:z_1z_2z_3}
Let $1/N \ll \eps \ll \gamma$. Let~$H$ be a $(1-\eps, \eps)$-dense $2$-edge-coloured $4$-graph on~$N$ vertices, let~$G$ be an $\eps$-blueprint for~$H$ with $\delta(G) \geq (1-\eps)N$. Let $W \subseteq V(G)$ be a set of size at least $\gamma N$. Let $T_1, T_2 \subseteq V(G)$ be sets such that, for $i \in [2]$, 
\begin{enumerate}[label = \textup{(\alph*)}]
    \item $2 \leq \s{T_i} \leq 4$,
    \item $T_i \in H^+$ if $\s{T_i} = 4$,
    \item $G[T_i] \cong K_{\s{T_i}}$,
    \item $N_H(S) \neq \varnothing$ for all $S \in \binom{T_i}{3}$.
\end{enumerate}
Then there exist vertices $z_1, z_2, z_3 \in W$ such that, for $i \in [2]$, 
\begin{enumerate}[label = \textup{(\roman*)}]
    \item $H[T_i \cup z_1z_2z_3] \cong K_{\s{T_i} + 3}^{(4)}$, 
    \item $G[T_i \cup z_1z_2z_3] \cong K_{\s{T_i} + 3}$,
    \item $xyz_1 \in \partial H(xy)$ for all $xy \in \binom{T_i}{2}$, 
    \item $xyz_2 \in \partial H(xy)$ for all $xy \in \binom{T_i \cup z_1}{2}$,
    \item $xyz_3 \in \partial H(xy)$ for all $xy \in \binom{T_i \cup z_1z_2}{2}$.
\end{enumerate}
In particular, $H^+[T_i \cup z_1z_2z_3] \cong K_{\s{T_i} + 3}^{(4)}$ for $i \in [2]$.
\end{prop}
\begin{proof}
Choose vertices 
\begin{align*}
    z_1 &\in W \cap \bigcap_{S \in \binom{T_1}{3} \cup \binom{T_2}{3}} N_H(S) \cap \bigcap_{x \in T_1 \cup T_2} N_G(x) \cap \bigcap_{xy \in \binom{T_1}{2} \cup \binom{T_2}{2}} N_{\partial H(xy)}(xy), \\
    z_2 &\in W \cap \bigcap_{S \in \binom{T_1 \cup z_1}{3} \cup \binom{T_2 \cup z_1}{3}} N_H(S) \cap \bigcap_{x \in T_1 \cup T_2 \cup z_1} N_G(x) \cap \bigcap_{xy \in \binom{T_1 \cup z_1}{2} \cup \binom{T_2 \cup z_1}{2}} N_{\partial H(xy)}(xy) \text{ and } \\
    z_3 &\in W \cap \bigcap_{S \in \binom{T_1 \cup z_1z_2}{3} \cup \binom{T_2 \cup z_1z_2}{3}} N_H(S) \cap \bigcap_{x \in T_1 \cup T_2 \cup z_1z_2} N_G(x) \cap \bigcap_{xy \in \binom{T_1 \cup z_1z_2}{2} \cup \binom{T_2 \cup z_1z_2}{2}} N_{\partial H(xy)}(xy),
\end{align*}
noting that these vertices exist since~$H$ is $(1-\eps, \eps)$-dense,\COMMENT{We use that for $x,y,z \in V(H)$ with $xyz \in \partial H$, we have $N_H(xyz) > 0$ and thus $N_H(xyz) \geq (1-\eps)N$.}~$G$ is an $\eps$-blueprint for~$H$ with $\delta(G) \geq (1-\eps)N$ and $\s{W} \geq \gamma N$. 
\end{proof}

The following corollary states that in our usual setting the good edges are tightly connected in any not too small induced subgraph of~$H$.

\begin{cor}
\label{cor:H^+_tightly_connected}
Let $1/N \ll \eps \ll \gamma$. Let~$H$ be a $(1-\eps, \eps)$-dense $2$-edge-coloured $4$-graph on~$N$ vertices. Let~$G$ be an $\eps$-blueprint for~$H$ with $\delta(G) \geq (1-\eps)N$. Let $W \subseteq V(G)$ be a set of size at least $\gamma N$.
Then~$H^+[W]$ is tightly connected.
\end{cor}
\begin{proof}
Let $f_1, f_2 \in H^+[W]$. By~\cref{prop:z_1z_2z_3}, there exist vertices $z_1,z_2,z_3 \in W$ such that $H^+[f_1 \cup z_1z_2z_3] \cong H^+[f_2 \cup z_1z_2z_3] \cong K_7^{(4)}$. It follows that~$f_1$ and~$f_2$ are in the same tight component of~$H^+$.
\end{proof}

The next lemma allows us to find blue tight components with useful properties. Recall that given a $2$-edge-coloured $4$-graph $H$, a blueprint $G$ for $H$ and a blue tight component $B$ of $H$, we denote by $B^2$ the set of edges  $e \in G^{\blue}$ such that $B(e) = B$.

\begin{lem}
\label{lem:B_W_exists_exact}
Let $1/n \ll \eps \ll \gamma \ll \eta <1$. Let $N =  (5/4 + 3 \eta)n $. Let~$H$ be a $2$-edge-coloured $(1-\eps, \eps)$-dense $4$-graph on~$N$ vertices and let~$G$ be an $\eps$-blueprint for~$H$ with $\delta(G) \geq (1- \eps)N$. Suppose~$H$ contains a red tight component~$R$ satisfying $H(e) = R$ for every $e \in G^{\red}$.
Then for each $W \subseteq V(G)$ with $\s{W} \geq \gamma n$ such that $R^+[W] = \varnothing$, there exists a blue tight component~$B_W$ of~$H$ such that the following holds. Let $\mathcal{T}_W = \{T \in \binom{W}{3} \colon G[T] \cong K_3, G^{\red}[T] \neq \varnothing, N_H(T) \neq \varnothing\}$. For $T \in \mathcal{T}_W$, let $\Gamma_W(T) = W \cap \bigcap_{x \in T} N_G(x) \cap \bigcap_{xy \in \binom{T}{2}} N_{\partial H(xy)}(xy) \cap N_H(T)$.
\begin{enumerate}[label = \textup{(B\arabic*)}, leftmargin= \widthof{B1000}]
    \item For any $T \in \mathcal{T}_W$, we have $\Gamma_W(T) \neq \varnothing$ and $T \cup w \in B_W^+$ for all $w \in \Gamma_W(T)$. In particular, $\mathcal{T}_W \subseteq \partial B_W$. \label{B1}
    \item For each $e \in G^{\blue}[W]$, we have $B(e) = B_W$, that is, $G^{\blue}[W] \subseteq B_W^2$. \label{B2}
\end{enumerate}
Moreover, if $W_1, W_2 \subseteq V(G)$ satisfy $\s{W_1 \cap W_2} \geq \gamma n$ and $R^+[W_1] = R^+[W_2] = \varnothing$, then $B_{W_1} = B_{W_2}$. 
\end{lem}
\begin{proof}
For $T \in \mathcal{T}_W$, since~$G$ is an $\eps$-blueprint for~$H$ with $\delta(G) \geq (1 - \eps)N$ and~$H$ is $(1-\eps, \eps)$-dense, we have 
\begin{align}
    \label{Gamma_bound}
    \s{\Gamma_W(T)} \geq \s{W} - 7\eps N \geq \s{W} - 14 \eps n.
\end{align}
In particular, for any $T \in \mathcal{T}_W$, $\Gamma_W(T) \neq \varnothing$ (since $\s{W} \geq \gamma n$ and $\eps \ll \gamma$).

Note that if $T \in \mathcal{T}_W$ and $w \in \Gamma_W(T)$, then $T \cup w \in H^{\blue}$ (or else $T \cup w \in R^+$ as $xyw \in \partial R$ for some $xy \in G^{\red}[T]$ contradicting $R^+[W] = \varnothing$).
For $T \in \mathcal{T}_W$, let~$B_W^T$ be the blue tight component of~$H$ containing all the edges $T \cup w$ where $w \in \Gamma_W(T)$. Note that, in particular, $T \in \partial B_W^T$ for every $T \in \mathcal{T}_W$. Moreover, 
\begin{align}
\label{B_W_property_1}
\text{$B(e) = B_W^T$ for all $T \in \mathcal{T}_W$ and $e \in G^{\blue}[T]$.}
\end{align} 

\begin{claim}
There exists a blue tight component~$B_W$ such that $B_W^T = B(e) = B_W$ for any $e \in G^{\blue}[W]$ and any $T \in \mathcal{T}_W$.
\end{claim}
\begin{proofclaim}
First assume that $\mathcal{T}_W = \varnothing$. This implies that $G^{\red}[W] = \varnothing$.
So $G^{\blue}[W]$ is connected and thus, since~$G$ is a blueprint, $B(e_1) = B(e_2)$ for any $e_1, e_2 \in G^{\blue}[W]$. So we may set $B_W = B(e)$ for all $e \in G^{\blue}[W]$. It follows that \cref{B1} and \cref{B2} both hold.

Now assume that $\mathcal{T}_W \neq \varnothing$.
First we show that 
\begin{align}
    \label{B_W_property_2}
    \text{$B_W^{T_1} = B_W^{T_2}$ for any $T_1, T_2 \in \mathcal{T}_W$ with $\s{T_1 \cap T_2} \geq 2$.}
\end{align}
Let $T_1, T_2 \in \mathcal{T}_W$ with $\s{T_1 \cap T_2} = 2$. By~\cref{Gamma_bound}, there exists $w \in \Gamma_W(T_1) \cap \Gamma_W(T_2)$. We have $T_1 \cup w \in B_W^{T_1}$ and $T_2 \cup w \in B_W^{T_2}$. Since $\s{(T_1 \cup w) \cap (T_2 \cup w)} = 3$ and~$B_W^{T_1}$ and~$B_W^{T_2}$ are blue tight components, we have $B_W^{T_1} = B_W^{T_2}$. 

Now we show that \cref{B_W_property_2} actually holds for any $T_1, T_2 \in \mathcal{T}_W$, that is
\begin{align}
    \label{B_W_property_3}
    \text{$B_W^{T_1} = B_W^{T_2}$ for any $T_1, T_2 \in \mathcal{T}_W$.}
\end{align}
Let $T_1, T_2 \in \mathcal{T}_W$. Say $T_1 = x_1x_2x_3$ and $T_2 = y_1y_2y_3$, where $x_1x_2 \in G^{\red}$ and $y_1y_2 \in G^{\red}$. 
By~\cref{prop:z_1z_2z_3}, there exist vertices $z_1, z_2 \in W$ such that 
\[
\text{$H[T_i \cup z_1z_2] \cong K_5^{(4)}$ and $G[T_i \cup z_1z_2] \cong K_5$ for $i \in [2]$.}
\]
Note that $x_1x_2z_1, y_1y_2z_1 \in \mathcal{T}_W$. If~$x_1z_1$ and~$y_1z_1$ are both in~$G^{\red}$, then $x_1z_1z_2, y_1z_1z_2 \in \mathcal{T}_W$ and thus by~\cref{B_W_property_2}, we have 
$B_W^{T_1} = B_W^{x_1x_2z_1} = B_W^{x_1z_1z_2} = B_W^{y_1z_1z_2} = B_W^{y_1y_2z_1} = B_W^{T_2}.$
If~$x_1z_1$ and~$y_1z_1$ are both in~$G^{\blue}$, then by~\cref{B_W_property_1}, \cref{B_W_property_2} and the fact that~$G$ is a blueprint, we have $B_W^{T_1} = B_W^{x_1x_2z_1} = B(x_1z_1) = B(y_1z_1) = B_W^{y_1y_2z_1} = B_W^{T_2}$. Now assume that exactly one of~$x_1z_1$ and~$y_1z_1$ is in~$G^{\red}$, say $x_1z_1 \in G^{\red}$ and $y_1z_1 \in G^{\blue}$. Note that $x_1z_1z_2 \in \mathcal{T}_W$. If $z_1z_2 \in G^{\red}$, then $x_1z_1z_2, y_1z_1z_2 \in \mathcal{T}_W$ and so by~\cref{B_W_property_1}, we have $B_W^{T_1} = B_W^{x_1x_2z_1} = B_W^{x_1z_1z_2} = B_W^{y_1z_1z_2} = B_W^{y_1y_2z_1} = B_W^{T_2}$. If $z_1z_2 \in G^{\blue}$, then by~\cref{B_W_property_1}, \cref{B_W_property_2} and the fact that~$G$ is a blueprint, we have $B_W^{T_1} = B_W^{x_1x_2z_1} = B_W^{x_1z_1z_2} = B(z_1z_2) = B(y_1z_1) = B_W^{y_1y_2z_1} = B_W^{T_2}$.

Now we show that 
\begin{align}
    \label{B_W_property_4}
    \text{$B_W^T = B(e)$ for any $T \in \mathcal{T}_W$ and any $e \in G^{\blue}[W]$.}
\end{align}
Let $T \in \mathcal{T}_W$ with $e_1 = x_1x_2 \in G^{\red}[T]$ and $e_2 = y_1y_2 \in G^{\blue}[W]$.
By~\cref{prop:z_1z_2z_3}, there exist vertices $z_1, z_2 \in W$ such that 
\[
\text{$e_i \cup z_1z_2 \in H^+$ for $i \in [2]$.}
\]
If $y_1z_1 \in G^{\red}$, then $y_1y_2z_1 \in \mathcal{T}_W$ and by~\cref{B_W_property_1} and \cref{B_W_property_3}, we have
$B_W^T = B_W^{y_1y_2z_1} = B(e_2)$.
Now assume $y_1z_1 \in G^{\blue}$. 
If $x_1z_1 \in G^{\blue}$, then by~\cref{B_W_property_1}, \cref{B_W_property_3} and the fact that~$G$ is a blueprint, we have 
$B_W^T = B_W^{x_1x_2z_1} = B(x_1z_1) = B(y_1z_1) = B(e_2)$.
Next assume $x_1z_1 \in G^{\red}$. 
If $z_1z_2 \in G^{\red}$, then by~\cref{B_W_property_1}, \cref{B_W_property_3} and the fact that~$G$ is a blueprint, we have 
$B_W^T = B_W^{y_1z_1z_2} = B(y_1z_1) = B(e_2)$.
If $z_1z_2 \in G^{\blue}$, then by~\cref{B_W_property_1}, \cref{B_W_property_3} and the fact that~$G$ is a blueprint, we have
$B_W^T = B_W^{x_1z_1z_2} = B(z_1z_2) = B(y_1z_1) = B(e_2)$.

Since $\mathcal{T}_W \neq \varnothing$, \cref{B_W_property_4} implies that
\begin{align}
    \label{B_W_property_5}
    \text{$B(e_1) = B(e_2)$ for any $e_1, e_2 \in G^{\blue}[W]$.}
\end{align}

It follows from \cref{B_W_property_3}, \cref{B_W_property_4} and \cref{B_W_property_5} that we may set $B_W = B_W^T = B(e)$ for all $T \in \mathcal{T}_W$ and all $e \in G^{\blue}[W]$. It follows that \cref{B1} and \cref{B2} hold.
\end{proofclaim}

Now we show the final statement of the lemma.
Let $W_1, W_2 \subseteq V(G)$ with $\s{W_1 \cap W_2} \geq \gamma n$ and $R^+[W_1] = R^+[W_2] = \varnothing$. Greedily choose vertices $x_1, x_2, x_3 \in W_1 \cap W_2$ such that $G[x_1x_2x_3] \cong K_3$ and $x_1x_2x_3 \in \partial H(x_1x_2)$ (this is possible since~$G$ is an $\eps$-blueprint with $\Delta(\overline{G}) \leq 2 \eps n$ and $1/n \ll \eps \ll \gamma$). Note that since $x_1x_2x_3 \in \partial H(x_1x_2)$, we have $N_H(x_1x_2x_3) \neq \varnothing$. If $x_1x_2 \in G^{\red}$, then $x_1x_2x_3 \in \mathcal{T}_{W_1} \cap \mathcal{T}_{W_2} \subseteq \partial B_{W_1} \cap \partial B_{W_2}$ and thus $B_{W_1} = B_{W_2}$. If $x_1x_2 \in G^{\blue}$, then $B_{W_1} = B(x_1x_2) = B_{W_2}$.
\end{proof}

We are now ready to prove \cref{lem:main_matchings_lemma} assuming \cref{lem:increase_matching}.

\begin{proof}[Proof of \cref{lem:main_matchings_lemma}]
Suppose for a contradiction that there does not exist a tightly connected fractional matching in~$H$ with weight at least~$n/4$ and all weights at least~$c$. Let~$r = \binom{9}{4}!$.
Choose new constants $\eps_0, \gamma$ and~$\delta$ such that $1/n \ll \eps \ll \eps_0 \ll c \ll \gamma \ll \delta \ll \eta$. By~\cref{lem:generalblueprint} and \cref{cor:Ecomp}, there exists an $\eps_0$-blueprint $G$ for~$H$ with $\delta(G) \geq (1-\eps_0)N$ that contains a spanning monochromatic component. Without loss of generality assume that~$G$ contains a spanning red component and let~$R$ be the unique red tight component of~$H$ such that $H(e) = R$ for every edge $e \in G^{\red}$. 
\begin{claim}
\label{claim:initial_matching}
There exists a good matching~$M$ of size at least $3 \delta n$ in~$H$ that is contained in~$R$ or in a blue tight component of~$H$.
\end{claim}
\begin{proofclaim}
Let~$M$ be a maximum good matching in~$R$ and let $W = V(G) \setminus V(M)$. It follows that $R^+[W] = \varnothing$. Moreover, we may assume that $\s{M} < 3 \delta n$ (or else we are done). Thus $\s{W} \geq \s{V(G)} - 12 \delta n \geq N - 13 \delta n$. Let $B = B_W$ be the blue tight component that exists by~\cref{lem:B_W_exists_exact}. 

\begin{enumerate}[label=\textbf{Case \Alph*:}, ref=\Alph*, wide, labelwidth=0pt, labelindent=0pt]
\item \textbf{\boldmath $G^{\red}[W]$ contains a matching of size at least $3 \delta n$.\unboldmath} Let $t = 3 \delta n$ and let $\{u_iv_i \colon i \in [t]\}$ be a matching in $G^{\red}[W]$. 
Let $\mathcal{T}_W$ and~$\Gamma_W$ be defined as in \cref{lem:B_W_exists_exact}. Since~$G$ is an $\eps_0$-blueprint for~$H$ with $\delta(G) \geq (1-\eps_0)N$, there exist distinct vertices $w_i \in W$, one for each $i \in [t]$, such that $u_iv_iw_i \in \mathcal{T}_W$. Since~$H$ is $(1- \eps, \eps)$-dense, there exist disjoint vertices $w_i' \in \Gamma_W(u_iv_iw_i)$, one for each $i \in [t]$. By~\cref{lem:B_W_exists_exact} \cref{B1}, $u_iv_iw_iw_i' \in B^+$ for each $i \in [t]$. It follows that $\{u_iv_iw_iw_i' \colon i \in [t]\}$ is a good matching of size $3\delta n$ in~$B$, as required.

\item \textbf{\boldmath $G^{\red}[W]$ does not contain a matching of size at least $3 \delta n$. \unboldmath} It follows that there exists a set $W' \subseteq W$ of size at least $\s{W} - 6\delta n \geq N -19\delta n$ such that $G[W'] \subseteq G^{\blue}$. By~\cref{lem:B_W_exists_exact} \cref{B2}, $G[W'] \subseteq B^2$. Let~$M'$ be a maximum matching in~$B^+[W']$. We may assume that $\s{M'} < 3 \delta n$ (or else we are done). Let $W'' = W' \setminus V(M')$ and so $\s{W''} \geq N -31 \delta n$. Note that by the maximality of~$M'$ and $G[W'] \subseteq B^2$, we have that $H^+[W''] \subseteq H^{\red}$. By~\cref{cor:H^+_tightly_connected}, there exists a red tight component~$R_*$ of~$H$ such that $H^+[W''] = R_*^+[W'']$. Thus by~\cref{prop:greedy_matching},~$R_*^+[W'']$ contains a matching of size at least 
\[
\frac{\s{W''}}{4} - \delta N \geq \frac{N - 31 \delta n}{4} - 2 \delta n \geq \frac{n}{4},
\]
a contradiction. \qedhere
\end{enumerate}
\end{proofclaim}
Let $0 \leq L \leq 1/\gamma$ be the largest integer such that the following holds. Let~$H_*$ be an $r^L$-blow-up of~$H$. Let $\mathcal{P} = \{V_x \colon x \in V(H)\}$ be the partition of~$V(H_*)$. Let~$G_*$ be the $\eps$-blueprint for~$H_*$ as defined in \cref{prop:G_*_is_blueprint}. Let~$R_*$ be the red tight component in~$H_*$ that is the $r^L$-blow-up of $R$.
Then~$H_*$ contains a matching~$M_*$ in~$R_*$ or in a blue tight component of~$H_*$ with $\s{M_*} \geq r^{L}(3 \delta n + L\gamma n)$ that is good for $(H_*, G_*)$.

By~\cref{claim:initial_matching}, we have that~$L$ is well-defined. Moreover, if $L \geq \frac{1}{4\gamma}$, then by~\cref{prop:blow-up_to_fractional}, we are done. Hence we may assume that $L < \frac{1}{4\gamma}$. 
Let $n_* = nr^L$ and $N_* = Nr^L = (5/4 + 3\eta)n_*$. Since~$H$ is $(1-\eps, \eps)$-dense, \cref{prop:H_*_is_dense} implies that~$H_*$ is $(1-2\eps, 2\eps)$-dense and hence also $(1-\eps_0, \eps_0)$-dense. By~\cref{prop:G_*_is_blueprint} and \cref{prop:G_*_min_degree},~$G_*$ is an $\eps_0$-blueprint for~$H_*$ with $\delta(G_*) \geq (1-\eps_0)nr^L = (1-\eps_0)n_*$ such that $H_*(e) = R_*$ for all $e \in G_*^{\red}$. We may further assume that~$H_*$ does not contain a monochromatic tightly connected matching of size at least~$n_*/4$ (or else we are done by~\cref{prop:blow-up_to_fractional}). We apply \cref{lem:increase_matching} with $r, n_*, \eps_0, \gamma, \delta, \eta, N_*, H_*, G_*, R_*, M_*$ playing the roles of $r, n, \eps, \gamma, \delta, \eta, N, H, G, R, M$. 
We deduce that~$H_*$ contains a $1/r$-fractional matching in~$R_*$ or in a blue tight component of~$H_*$ that is good for $(H_*, G_*)$ of weight at least $\s{M_*} + \gamma n_*$. Let~$H_{**}$ be an $r$-blow-up of~$H_*$ and note that~$H_{**}$ is an $r^{L+1}$-blow-up of~$H$. Let $R_{**}$ be the red tight component of $H_{**}$ that is the $r$-blow-up of $R_*$ and thus is the $r^{L+1}$-blow-up of $R$. By~\cref{prop:G_*_is_blueprint}, $H_{**}(e) = R_{**}$ for all $e \in G_{**}^{\red}$. By~\cref{prop:make_integral},~$H_{**}$ contains a matching in~$R_{**}$ or in a blue tight component of~$H_{**}$ that is good for $(H_{**}, G_{**})$ of size at least 
\[
r(\s{M_*} + \gamma n_*) \geq r(r^{L}(3 \delta n + L\gamma n) +\gamma r^Ln) = r^{L+1}(3 \delta n + (L+1)\gamma n).
\]
This is a contradiction to the maximality of~$L$.
\end{proof}

\subsection{Sketch of the proof of Lemma~\ref{lem:increase_matching} and suitable pairs}

Before proceeding with the proof of \cref{lem:increase_matching}, we give a sketch of the proof. Recall that our aim is given a monochromatic tightly connected matching~$M$, to find a larger monochromatic tightly connected fractional matching. We split the proof into two cases depending on whether \cref{H1} or \cref{H2} holds. We only sketch the proof of the case where \cref{H1} holds, that is $M \subseteq R$ (the other case is similar). We may assume that~$M$ is a matching of maximum size in~$R$. Moreover, for simplicity, we assume that
\begin{align} \label{Assumption:A}
        \parbox{0.85\textwidth}{$H$ and~$G$ are complete, $V(G) = V(H)$ and for every $e \in G$ and $v \in V(H) \setminus e$, we have $e \cup v \in \partial H(e)$,} \tag{\textup{A}}
\end{align}
the last of which is an idealised version of \cref{BP1}.

Let $W = V(H) \setminus V(M)$ be the vertices of~$H$ not covered by~$M$. By \cref{lem:B_W_exists_exact}, there exists a blue tight component~$B$ in~$H$ such that $\binom{W}{3} \subseteq \partial B$ and every $e \in G^{\blue}[W]$ induces~$B$. We will find our desired fractional matching in~$R$ or~$B$.

Our first step is to find a subset $U' \subseteq W$ and two matchings $M' = \{f'(u) \colon u \in U'\} \subseteq M \subseteq R$ and
$M^*_1 = \{f_u^* \colon u \in U'\} \subseteq B$ such that $\s{U'} \approx \s{M}$, and 
\begin{align*}
    \s{f'(u) \cap f^*_u} = 3 \text{ and } f_u^* \setminus f'(u) = u \text{ for all } u \in U'.
\end{align*}
If no such~$U'$ exists, then we can find a small matching $M'' \subseteq M$ and for each $f \in M''$ a disjoint $4$-set $W_f \subseteq W$ such that, for each $f \in M''$ there exists a fractional matching~$\varphi_f$ in $R[f \cup W_f]$ of weight at least $\frac{r+1}{r}$ (where~$r$ is some absolute constant). By starting with~$M$ and replacing each edge $f \in M''$ with~$\varphi_f$ we get a larger fractional matching. (See \cref{claim:U'_large} for the details.)

In our second step we then extend the matching~$M_1^*$ to a larger fractional matching in~$B$ completing the proof.

We will use the following fact which allows us to find a fractional matching in $R[f \cup W_f]$.

\begin{fact} 
\label{fact:A}
Let $k,s \geq 2$. Let~$H$ be a $k$-graph and let $F \subseteq E(H)$ be a nonempty set such that $\bigcap F = \varnothing$ and $\s{F} = s$. Then there exists a $\frac{1}{s-1}$-fractional matching in~$H$ with weight~$\frac{s}{s-1}$.
\end{fact}
\begin{proof}
Let $\varphi \colon E(H) \rightarrow [0,1]$ be defined by $\varphi(e) = \frac{1}{s-1}$ for each $e \in F$ and $\varphi(e) = 0$ otherwise. Since $\bigcap F = \varnothing$ each vertex of~$H$ is contained in at most~$s-1$ edges of~$F$. It follows that for every $v \in V(H)$, we have $\sum_{e \in E(H)\colon v \in e} \varphi(e) \leq 1$. Thus~$\varphi$ is a $\frac{1}{s-1}$-fractional matching in~$H$.
\end{proof}

In the actual proof of \cref{lem:increase_matching}, $H$ and $G$ will only be almost complete and so do not satisfy \cref{Assumption:A}. In particular, we may not have $e \cup v \in \partial H(e)$ for some $e \in G$ and $v \in V(H) \setminus e$. 
To overcome the difficulties that arise form this, we introduce the following notion of \emph{suitable pairs}. For a suitable pair $(f, W)$, it is useful to think of~$f$ as an edge of some good matching and~$W$ as a subset of the vertices not covered by that matching. Then \cref{SP1,SP2,SP3,SP4,SP5,SP6} are the properties that we would have in the idealised case where \cref{Assumption:A} holds that are necessary for our proof. 

\begin{defn}[Suitable pairs]
Let~$H$ be a $2$-edge-coloured $4$-graph and~$G$ a blueprint for~$H$. Let $f \in H$ be a good edge and $W \subseteq V(G) \setminus f$. We call $(f, W)$ a \emph{suitable pair for $(H, G)$} if the following properties hold, where $s = \s{W}$.
\begin{enumerate}[label = \textup{(SP\arabic*)}, leftmargin= \widthof{SP1000}]
    \item $H[f \cup W] \cong K_{s+4}^{(4)}$, \label{SP1}
    \item $G[f \cup W] \cong K_{s+4}$, \label{SP2}
    \item if $xy \in \binom{f}{2}$ and $z \in W$, then $xyz \in \partial H(xy)$, \label{SP3}
    \item if $xy \in \binom{W}{2}$ and $z \in f$, then $xyz \in \partial H(xy)$, \label{SP4}
    \item if $x \in f$ and $yz \in \binom{W}{2}$, then $xyz \in \partial H(xy)$ and \label{SP5}
    \item if $xyz \in \binom{W}{3}$, then $xyz \in \partial H(xy)$. \label{SP6}
\end{enumerate}
If~$H$ and~$G$ are clear from context, we simply call $(f, W)$ a suitable pair.
Note that if $(f, W)$ is a suitable pair and $W' \subseteq W$, then $(f, W')$ is a suitable pair. Moreover, if $(f, W)$ is a suitable pair, then any edge in $H[f \cup W]$ is good. Also if $e \in R^2[W]$, then $f' \in R^+$ for any edge $f' \in H[f \cup W]$ with $e \subseteq f'$.
\end{defn}

We use the following lemma to find suitable pairs. The main idea is that if we choose uniformly at random an edge~$f$ from a good matching~$M$ and a subset~$W_f$ of constant size from $V(G) \setminus V(M)$, then $(f, W_f)$ is likely to be a suitable pair.

\begin{lem}
\label{lem:make_suitable}
Let $1/N \ll \eps \ll \gamma \ll \delta \ll 1/s \leq 1$. Let~$H$ be a $(1-\eps, \eps)$-dense $2$-edge-coloured $4$-graph on~$N$ vertices and let~$G$ be an $\eps$-blueprint for~$H$ with $\delta(G) \geq (1- \eps) N$. Let~$M$ be a good matching in~$H$ of size at least $\delta N$ and let $W \subseteq V(G) \setminus V(M)$ with $\s{W} \geq \delta N$. Then there exist a matching $M' \subseteq M$ with $\s{M'} \geq \gamma N$ and disjoint sets $W_f \in \binom{W}{s}$ for each $f \in M'$ such that $(f, W_f)$ is a suitable pair for each $f \in M'$.
\end{lem}
\begin{proof}
Let $m =  2\gamma N$. 
We independently choose $\{f_i \colon i \in [m]\}$ uniformly at random among all subsets of~$M$ of size~$m$ and $\{W_{f_i} \colon i \in [m]\}$ uniformly at random among all sets of~$m$ disjoint sets in $\binom{W}{s}$.
Note that for each $i \in [m]$,~$f_i$ is distributed uniformly in~$M$ and~$W_{f_i}$ is independent from~$f_i$ and distributed uniformly in $\binom{W}{s}$.
For each $i \in [m]$, let~$A_i$ be the event that $(f_i, W_{f_i})$ is a suitable pair. Note that it suffices to show that $\mathbb{P}[A_i] \geq 1/2$ for each $i \in [m]$ since then for $M' = \{f_i \colon i \in [m] \text{ such that~$A_i$ holds} \}$, we have $\mathbb{E}[\s{M'}] \geq m/2 = \gamma N$.

Fix $i \in [m]$. For $j \in \{1,2\}$, let~$A_{i,j}$ be the event that $(f_i, W_{f_i})$ satisfies $(\textup{SP}j)$. For $j \in \{3, \dots, 6\}$, we modify the statements slightly for the following probability calculation. For $j \in \{3, \dots, 6\}$, let~$A_{i,j}$ be the event that $(f_i, W_{f_i})$ satisfies $(\textup{SP}j)$ but only for those pairs~$xy$ such that $xy \in G$. Note that $A_i \subseteq \bigcap_{j \in [6]} A_{i,j}$.

To prove the lemma, it suffices to show that $\mathbb{P}[A_{i,j}] \geq 11/12$ for each $j \in [6]$. 

To bound $\mathbb{P}[A_{i,1}]$, we fix~$f_i$ and count the number of sets $W_{i} \in \binom{W}{s}$ such that $(f_i, W_{i})$ satisfies~\cref{SP1}. Note that if we iteratively choose \[
w_j \in W \cap \bigcap_{S \in \binom{f_i \cup w_1 \dots w_{j-1}}{3}} N_H(S)
\]
for each $j \in [s]$, then $(f_i, \{w_1, \dots, w_s\})$ satisfies~\cref{SP1}. Since $d_H(S) \geq (1- \eps)N$ for each $S \in \binom{V(H)}{3}$ with $N_H(S) \neq \varnothing$, the number of choices for~$w_j$ is at least $\s{W} - \binom{s+3}{3}\eps N$.
Hence the number of sets $W_{i} \in \binom{W}{s}$ such that $(f_i, W_{i})$ satisfies~\cref{SP1} is at least 
\[
\frac{(\s{W} - \binom{s+3}{3}\eps N)^s}{s!}. 
\]
It follows that
\begin{align*}
    \mathbb{P}[A_{i,1}] &\geq \frac{(\s{W} - \binom{s+3}{3}\eps N)^s}{\binom{\s{W}}{s}s!} 
    \geq \br{1- \binom{s+3}{3} \sqrt{\eps}}^s \geq 1- s\binom{s+3}{3} \sqrt{\eps} \geq \frac{11}{12},
\end{align*}
where in the second inequality we used that $\s{W} \geq \delta N$. 
By a similar argument, since $W \subseteq V(G)$ and $\delta(G) \geq (1-\eps)N$, we have 
\begin{align*}
    \mathbb{P}[A_{i,2}] \geq \frac{(\s{W} -(s+3) \eps N)^s}{\binom{\s{W}}{s}s!} \geq 1 - s(s+3) \sqrt{\eps} \geq \frac{11}{12}.
\end{align*}

To bound $\mathbb{P}[A_{i,3}]$, we fix~$f_i$ and recall that since $W_{f_i} = \{w_1, \dots, w_s\}$ is chosen uniformly at random in $\binom{W}{s}$, each~$w_j$ is distributed uniformly in~$W$. Since~$G$ is an $\eps$-blueprint, \cref{BP1} implies that $d_{\partial H(xy)}(xy) \geq (1- \eps)N$ for every $xy \in G$. Hence for $xy \in f_i$ and $j \in [s]$, we have that $\mathbb{P}[w_i \not\in N_{\partial H(xy)}(xy)] \leq \eps N / \s{W} \leq \sqrt{\eps}.$
A union bound implies that
\begin{align*}
    \mathbb{P}[A_{i,3}] \geq 1- 6s\sqrt{\eps} \geq \frac{11}{12}.
\end{align*}

To bound $\mathbb{P}[A_{i,4}]$ we fix~$W_{f_i}$ and recall that~$f_i$ is distributed uniformly in~$M$. Since~$G$ is an $\eps$-blueprint, for each $xy \in G[W_{f_i}]$, we have $d_{\partial H(xy)}(xy) \geq (1- \eps)N$. Hence there are at most $4\binom{s}{2}\eps N$ elements of~$M$ for which $(f_i, W_{f_i})$ does not satisfy~\cref{SP4}. 
It follows that
\begin{align*}
    \mathbb{P}[A_{i,4}] \geq \frac{\s{M}-4\binom{s}{2} \eps N}{\s{M}} \geq 1 - 4 \binom{s}{2} \sqrt{\eps} \geq \frac{11}{12},
\end{align*}
as $\s{M} \geq \delta N$.

To bound $\mathbb{P}[A_{i,5}]$, we fix~$f_i$ and let $W_{f_i} = \{w_1, \dots, w_s\}$. For each distinct $j, j' \in [s]$,~$w_jw_{j'}$ is distributed uniformly in $\binom{W}{2}$. Since~$G$ is an $\eps$-blueprint, for each $xy \in G$, we have $d_{\partial H(xy)}(xy) \geq (1 - \eps) N$. Hence, for $x \in f_i$ and distinct $j, j' \in [s]$, we have 
\[
\mathbb{P}[xw_j \not \in G \text{ or } w_{j'} \not \in N_{\partial H(xw_j)}(xw_j)] \leq 2 \cdot \frac{\s{W} \eps N}{2 \binom{\s{W}}{2}}.
\]
A union bound implies that
\begin{align*}
    \mathbb{P}[A_{i,5}] \geq 1 - 4 s^2 \frac{\s{W} \eps N}{ \binom{\s{W}}{2}} \geq 1 - 8 s^2 \sqrt{\eps} \geq \frac{11}{12}.
\end{align*}
Similarly, 
\begin{align*}
    \mathbb{P}[A_{i,6}] \geq 1 - s^3 \frac{\s{W}(\s{W} -1)\eps N}{3! \binom{\s{W}}{3}} \geq 1 - s^3 \sqrt{\eps} \geq \frac{11}{12}.
\end{align*}
\end{proof}

The following proposition analyses a specific pattern that we will encounter a few times in our proof.

\begin{prop}
\label{prop:local}
Let~$H$ be a $2$-edge-coloured $4$-graph and let~$G$ be a blueprint for~$H$. Let~$R$ be a red tight component of~$H$. Let $f \in R^+$ and let $(f, W)$ be a suitable pair such that $\s{W} =3$ and there is an edge $e \in R^2[W]$. 
Let $F = \{f' \in H^{\red}[f \cup W] \colon e \subseteq f'\}$.
Suppose $\bigcap R[f \cup W] \neq \varnothing$. Then 
there exists $x \in f \cap \bigcap F$. In particular, for any edge $f' \in H[f \cup W]$ with $e \subseteq f'$ and $x \not\in f'$, we have $f' \in H^{\blue}$. 
The same statement holds with colours reversed.
\end{prop}
\begin{proof}
Suppose for a contradiction that $f \cap \bigcap F = \varnothing$.  
Let $f' \in F$ and let $z \in f' \cap f$. By~\cref{SP4}, $e \cup z \in \partial H(e) = \partial R$.
Since $e \cup z \subseteq f'$ and $f' \in H^{\red}$, we have $f' \in R$. Thus $F \subseteq R$. It follows that $\bigcap R[f \cup W] = \varnothing$, a contradiction.
\end{proof}


\subsection{Proof of \texorpdfstring{\cref{lem:increase_matching} assuming \cref{H1}}{Lemma 4.4 assuming (H1)}}

We prove \cref{lem:increase_matching} for the case that \cref{H1} holds, that is, $M \subseteq R$.

\begin{proof}[Proof of \cref{lem:increase_matching} assuming \cref{H1}]
Assume for a contradiction that~$H$ does not contain a good $1/r$-fractional matching in~$R$ or in a blue tight component of~$H$ of weight at least $\s{M} + \gamma n$. 
Note that $\s{V(G)} \geq (1-\eps)N \geq N - 2\eps n \geq (5/4 +2\eta)n$. We will construct our fractional matching in~$H[V(G)]$ ignoring the small number of vertices in $V(H) \setminus V(G)$.

It suffices to assume that~$M$ is a maximum good matching in~$R$ (that is a maximum matching in~$R^+$) and that among all such matchings~$M$ contains the smallest number of edges~$f$ such that 
\begin{align}
    \label{M_special_condition}
    \text{$G^{\blue}[f] \neq \varnothing$ and $B(e) \neq B_W$ for all $e \in G^{\blue}[f]$,}
\end{align}
where $W = W(M)= V(G) \setminus V(M)$ and~$B_W$ is as in \cref{lem:B_W_exists_exact}. Here~$B_W$ is defined since $R^+[W] = \varnothing$ by the maximality of~$M$ and $\s{W} \geq (5/4 + 2\eta)n - 4 \s{M} \geq (1/4 + 2\eta)n \geq \gamma n$ as $\s{M} < n/4$.
Let $B = B_W$. 

\begin{claim}
\label{claim:B^2_in_f_exact}
Let $f \in M$ and $e \in G^{\blue}[W]$ such that $G^{\blue}[f] \neq \varnothing$ and $(f, e)$ is a suitable pair.
Then~$f$ contains an edge $e' \in G^{\blue}[f]$ with $B(e') = B$.
\end{claim}
\begin{proofclaim}
Suppose for a contradiction that $B(e') \neq B$ for all $e' \in G^{\blue}[f]$. Let $e' = x_1y_1 \in G^{\blue}[f]$ and $e =x_2y_2$. By~\cref{lem:B_W_exists_exact}, $B(e) = B$. Since~$G$ is a blueprint and $B(e') \neq B$, we have $x_1x_2 \in G^{\red}$. By~\cref{SP5}, $x_1x_2y_2 \in \partial R$. If $f_* =x_1x_2y_1y_2 \in H^{\red}$, then $f_* \in R^+$ and thus the matching $M^* = (M \setminus \{f\}) \cup \{f_*\}$ is good and has one less edge satisfying \cref{M_special_condition} than~$M$ (since $B_{W(M^*)} = B_{W(M)}$ by~\cref{lem:B_W_exists_exact}), a contradiction. Hence $f_* \in H^{\blue}$. By~\cref{SP4}, $x_1x_2y_2 \in \partial H(x_2y_2) = \partial B$ and by~\cref{SP3}, $x_1y_1x_2 \in \partial H(x_1y_1) = \partial B(e')$. It follows that $B(e') = B$, a contradiction.
\end{proofclaim}

Let $U \subseteq W$ be a set of maximum size such that
for each $u \in U$, there exists a distinct edge $f(u) \in M$ so that $(f(u), u)$ is a suitable pair and $H^{\red}[f(u) \cup u] \cong K_5^{(4)}$.
If $\s{U} \geq 4\gamma n$, then we are done since the $1/r$-fractional matching $\varphi \colon R \rightarrow [0,1]$ with $\varphi(e) = 1$ for $e \in M \setminus \{f(u) \colon u \in U\}$,  $\varphi(e) = 1/4$ for each edge $e \in \bigcup_{u \in U} H^{\red}[f(u) \cup u]$ and $\varphi(e) = 0$ for all other edges $e \in R$ is a good $1/r$-fractional matching in~$R$ of weight at least $\s{M} + \gamma n$. Now assume that $\s{U} < 4 \gamma n$. Let $W' = W \setminus U$ and $M' = M \setminus \{f(u) \colon u \in U\}$. Note that $\s{W'} \geq (5/4 + 2\eta)n - 4\s{M} - 4\gamma n \geq (1/4 + \eta) n$ and $2 \delta n \leq \s{M} - 4\gamma n \leq \s{M'} \leq n/4$. 

By the maximality of~$U$, we have that
\begin{equation}
\label{Uproperty_exact}
H^{\blue}[f \cup w] \neq \varnothing
\end{equation}
for every suitable pair $(f, w) \in M' \times W'$.
Let $U' \subseteq W'$ be a set of maximum size such that there exists for each $u \in U'$, a distinct edge $f'(u) \in M'$ such that $(f'(u), u)$ is a suitable pair and $B[f'(u) \cup u] \neq \varnothing$. Let 
\[
\text{$W'' = W' \setminus U'$ and $M'' = M' \setminus \{f'(u) \colon u \in U'\}$.}
\]
Note that $\s{W''} \geq \s{W'} -  \s{U'} \geq (1/4 + \eta)n - \s{M'} \geq \eta n \geq \eta N/2$. 
Let $\gamma \ll \delta_0 \ll \delta$.
\begin{claim}
\label{claim:U'_large}
We have $\s{U'} \geq \s{M'} - \delta_0 n$. 
\end{claim}
\begin{proofclaim}
Suppose not. We have $\s{M''} \geq \delta_0 n \geq \delta_0 N/2$.
By the maximality of~$U'$, we have 
\begin{equation}
\label{U'property_exact}
B[f \cup w] = \varnothing
\end{equation}
for every suitable pair $(f, w) \in M'' \times W''$.
By~\cref{lem:make_suitable}, there exists $M^* \subseteq M''$ with $\s{M^*} = r \gamma n$ and disjoint sets $W_f \in \binom{W''}{3}$ for each $f \in M^*$ such that $(f, W_f)$ is a suitable pair for each $f \in M^*$. 
Let~$\varphi_0$ be the fractional matching induced by the matching $M \setminus M^*$.
It suffices to show that, for every $f \in M^*$, there exists a $1/r$-fractional matching~$\varphi_f$ in $R[f \cup W_f]$ of weight at least~$\frac{r+1}{r}$.
Indeed, the completion of $\varphi_0 + \sum_{f \in M^*} \varphi_f$ with respect to~$R$ is a good $1/r$-fractional matching in~$R$ of weight at least $\s{M \setminus M^*} + \frac{r+1}{r}\s{M^*} \geq \s{M} + \gamma n$ giving us a contradiction. 

Consider any $f = x_1x_2x_3x_4 \in M^*$.
By~\cref{fact:A} and since $r = \binom{9}{4}!$, we may assume that $\bigcap R[f \cup W_f] \neq \varnothing$.
We distinguish between several cases.

\begin{enumerate}[label=\textbf{Case \Alph*:}, ref=\Alph*, wide, labelwidth=0pt, labelindent=0pt, topsep=5pt]
\item \textbf{\boldmath $G^{\red}[W_f] \neq \varnothing$.\unboldmath}
Let $W_f = uvw$ with $e = uv \in G^{\red}[W_f]$. Recall that $(f, W_f)$ is a suitable pair.
We apply \cref{prop:local} with $r, H, G, R, f, W_f, e$ playing the roles of $r, H, G, R_*, f, W, e$. So there exists $x \in f$ such that 
\begin{align}
    \label{Local_application_1}
    \text{$f' \in H^{\blue}$ for any edge $f' \in H[f \cup W_f]$ with $e \subseteq f'$ and $x \not\in f'$.}
\end{align}
By~\cref{Uproperty_exact} and \cref{U'property_exact}, there exists an edge $f_* \in H^{\blue}[f \cup u] \setminus B$. Since $f \in R$, we have $u \in f_*$. Let $yz \subseteq (f_* \cap f) \setminus x$. By~\cref{Local_application_1}, $yuvw, yzuv \in H^{\blue}$.
By~\cref{lem:B_W_exists_exact} \cref{B1}, $uvw \in \partial B$. Hence $yuvw, yzuv, f_* \in B$, a contradiction to $f_* \not\in B$.

\item \textbf{\boldmath $G^{\red}[W_f] = \varnothing$.\unboldmath}
By~\cref{lem:B_W_exists_exact} \cref{B2}, we have $G[W_f] \subseteq G^{\blue} \subseteq B^2$.
Let $uv \in B^2[W_f]$. We distinguish between the following cases. It is easy to see that these cases exhaust all possibilities.

\begin{enumerate}[label=\textbf{Case \Alph{enumi}.\arabic*:},ref = \Alph{enumi}.\arabic*, wide, labelwidth=0pt, labelindent=0pt]
\item \textbf{\boldmath $\s{G^{\blue}[f]} \geq 3$ and $G^{\blue}[f] \ncong K_{1,3}$.\unboldmath}
Note that $G^{\blue}[f]$ is connected, hence \cref{claim:B^2_in_f_exact} and the fact that~$G$ is a blueprint imply that $G^{\blue}[f] \subseteq B^2$. 
By~\cref{Uproperty_exact}, there exists an edge $f_* \in H^{\blue}[f \cup u]$.
Observe that~$f_*$ contains an edge $xy \in G^{\blue}[f]$ and $u \in f_*$. By~\cref{SP3}, we have $xyu \in \partial H(xy) = \partial B$. Hence $f_* \in B$, a contradiction to \cref{U'property_exact}.

\item \textbf{\boldmath $\s{G^{\blue}[f]} \geq 2$ and $\bigcap G^{\blue}[f] \neq \varnothing$.\unboldmath}
Without loss of generality assume that $x_1x_2, x_1x_3 \in G^{\blue}[f]$ and $x_2x_3, x_2x_4, x_3x_4 \in G^{\red}$. By~\cref{claim:B^2_in_f_exact} and the fact that~$G$ is a blueprint, we have that $G^{\blue}[f] \subseteq B^2$. 
By~\cref{Uproperty_exact} and \cref{U'property_exact}, there exists an edge $f_* \in H^{\blue}[f \cup u] \setminus B$. Observe that $f_* = x_2x_3x_4u$ since, by~\cref{SP3}, $x_1x_2u, x_1x_3u \in \partial B$. Let $e_{f,1} = x_2x_3uv$, $e_{f,2} = x_2x_4uv$, $e_{f,3} = x_3x_4uv$. Note that, for all $i \in [3]$, we have $e_{f,i} \cap f \in R^2$ and $e_{f,i} \cap W_f \in B^2$ and thus, by~\cref{SP3} and~\cref{SP4}, we have $e_{f,i} \in R \cup B$. Since $f_* \not\in B$, we have $e_{f,i} \in R$ for all $i \in [3]$.
We are done since $\{f, e_{f,1}, e_{f,2}, e_{f,3}\} \subseteq R[f \cup W_f]$ has an empty intersection, a contradiction.

\item \textbf{\boldmath $G^{\red}[f]$ contains a copy of~$C_4$.\unboldmath}
Without loss of generality assume that $x_1x_2, x_2x_3, x_3x_4, x_1x_4 \in G^{\red}[f]$. By~\cref{SP4}, we have $uvx_j \in \partial B$ for all $j \in [4]$. By~\cref{SP3}, we have $x_1x_2u, x_2x_3u, x_3x_4u, x_1x_4u \in \partial R$ for all $i \in [4]$. 

If~$x_1x_2uv$ and~$x_3x_4uv$ are red, then $F = \{f, x_1x_2uv, x_3x_4uv\} \subseteq R$ has an empty intersection, a contradiction.
So we may assume that~$x_1x_2uv$ is blue and thus in~$B$. Similarly, by considering $\{f, x_1x_4uv, x_2x_3uv\}$, we may assume that $x_1x_4uv \in B$.
By~\cref{Uproperty_exact} and \cref{U'property_exact}, there exists an edge $f_* \in H^{\blue}[f \cup u] \setminus B$. Since $x_1x_2uv, x_1x_4uv \in B$, we have $f_* = x_2x_3x_4u$ and $x_2x_3uv, x_2x_4uv, x_3x_4uv \in R$. Thus we obtain a contradiction as $\{f, x_2x_3uv, x_2x_4uv, x_3x_4uv\} \subseteq R$ has an empty intersection.
\qedhere
\end{enumerate}
\end{enumerate}
\end{proofclaim}

For the remainder of the proof, our aim is to find a good $1/r$-fractional matching in~$B$ of weight at least $\s{M} + \gamma n$.
For each $u \in U'$, choose an edge $f_u^* \in B[f'(u) \cup u]$ which exists by the definition of~$U'$. Since $(f'(u), u)$ is a suitable pair,~$f_u^*$ is good. Let $M_1^* = \{f_u^* \colon u \in U'\}$ and note that~$M_1^*$ is a good matching in~$B$.
Note 
\[
\s{M_1^*} = \s{U'} \geq \s{M'} - \delta_0 n \geq \s{M} - 2\delta_0 n.
\]
Let $M_2^* \subseteq B^+[V(G) \setminus V(M_1^*)]$ be a maximum matching. If $\s{M_1^*} + \s{M_2^*} \geq \s{M} + \gamma n$, then we are done. Thus we may assume $\s{M_1^*} + \s{M_2^*}< \s{M} + \gamma n$, so $\s{M_2^*} \leq 3 \delta_0 n$. 
Let
\begin{align*}
    U'' &= \{u \in U' \colon (f'(u) \cup u) \cap V(M_2^*) = \varnothing\}, \\
    M_0 &= \{f'(u)\colon u \in U''\} \text{ and } \\
    W_0 &= W'' \setminus V(M_2^*) \subseteq W'.
\end{align*}
We have
\begin{align*}
\begin{split}
    \s{U''} &\geq \s{U'} - 4\s{M_2^*} \geq \s{M} - 14\delta_0 n \geq 2\delta n \geq \delta N, \\
    \s{M_0} &= \s{U''} \geq \delta N \text{ and } \\
    \s{W_0} &\geq \s{W''} - 4 \s{M_2^*} \geq \eta n/2 \geq \eta N/4.
\end{split}
\end{align*}
By~\cref{lem:make_suitable}, there exist a subset $U_0 \subseteq U''$ corresponding to the matching $\{f'(u)\colon u \in U_0\} \subseteq M_0$ of size $3 r \delta_0 n$ and disjoint sets $W_u \in \binom{W_0}{4}$ for each $u \in U_0$ such that $(f'(u), W_u)$ is a suitable pair for each $u \in U_0$.

We now construct a good $1/r$-fractional matching $\varphi \colon B \rightarrow [0,1]$ in~$B$ as follows. 
Let~$\varphi_0$ be the fractional matching induced by the matching $(M_1^* \setminus \{f_u^* \colon u \in U_0\}) \cup M_2^*$. Suppose that, for each $u \in U_0$, there exists a good $1/r$-fractional matching~$\varphi_u$ in $B[f'(u) \cup u \cup W_u]$ of weight at least~$\frac{r+1}{r}$. Then the completion of $\varphi_0 + \sum_{u \in U_0} \varphi_u$ with respect to~$B$ is a good $1/r$-fractional matching in~$B$ of weight at least $\s{M_1^*} + \s{M_2^*} + \s{U_0}/r \geq \s{M} - 2\delta_0 n + 3 \delta_0 n \geq \s{M} + \gamma n$. Thus it suffices to show that, for each $u \in U_0$, there exists a good $1/r$-fractional matching~$\varphi_u$ in $B[f'(u) \cup u \cup W_u]$ of weight at least~$\frac{r+1}{r}$.
Note that $B[f'(u) \cup W_u] \cup \{f_u^*\} \subseteq B^+$. By~\cref{fact:A}, it suffices to show that $\bigcap (B[f'(u) \cup W_u] \cup \{f_u^*\}) = \varnothing$.

Consider any $u \in U_0$.
Let 
\[
\text{$f'(u) = yz_1z_2z_3 \in R$, $f_u^* = z_1z_2z_3u \in B$ and $W_u = w_1w_2w_3w_4$.}
\]
By the maximality of~$M_2^*$, we have $w_1w_2w_3w_4 \not\in B$. Hence \cref{B1} implies that $G^{\red}[W_u] = \varnothing$. Thus by~\cref{lem:B_W_exists_exact} \cref{B2}, we have $G[W_u] \subseteq B^2$. In particular, $w_1w_2 \in B^2$ and thus~\cref{SP4} and~\cref{SP6} imply $yw_1w_2, w_1w_2w_3 \in \partial B$. By the maximality of~$M_2^*$ and the maximality of~$M$, we have that $yw_1w_2w_3, w_1w_2w_3w_4 \in H^{\red} \setminus R$.\COMMENT{Since $yz_1z_2z_3 \in R$ and~$R$ is a red tight component, this implies that we have $y z_{i_1} z_{i_2} w_{i_3} \in H^{\blue}$ or $y z_{i_1} w_{i_2} w_{i_3} \in H^{\blue}$ for all distinct $i_1, i_2, i_3 \in [4]$.}

We distinguish between two cases. 
\begin{enumerate}[label=\textbf{Case \Alph*:}, ref=\Alph*, wide, labelwidth=0pt, labelindent=0pt]
\item \textbf{\boldmath At least two of $yz_1, yz_2, yz_3$ are in~$G^{\red}$.\unboldmath}\label{Case_A_exact}
Without loss of generality assume that $yz_1, yz_2 \in G^{\red}$. By~\cref{SP3}, we have $yz_1w_1, yz_2w_1 \in \partial R$. Since $yw_1w_2w_3 \in H^{\red} \setminus R$ and $yw_1w_2 \in \partial B$, we have $yz_1w_1w_2, yz_2w_1w_2 \in B$. Thus we are done since $\{z_1z_2z_3u, yz_1w_1w_2, yz_2w_1w_2\} \subseteq B$ has an empty intersection.

\item \textbf{\boldmath At least two of $yz_1, yz_2, yz_3$ are in~$G^{\blue}$.\unboldmath}
Without loss of generality assume that $yz_1, yz_2 \in G^{\blue}$, so the edges of $G^{\blue}[yz_1z_2z_3]$ form a connected component. Since $w_1w_2 \in G^{\blue}[W]$ and $(yz_1z_2z_3, w_1w_2)$ is a suitable pair, \cref{claim:B^2_in_f_exact} implies that $B(yz_1) = B(yz_2) = B$ and thus $yz_1w_1, yz_2w_1 \in \partial B$ by~\cref{SP3}. If $yz_1w_1w_2, yz_2w_1w_2 \in H^{\blue}$, then $yz_1w_1w_2, yz_2w_1w_2 \in B$ (since $w_1w_2y \in \partial B$ by~\cref{SP4}). Note that $\{z_1z_2z_3u, yz_1w_1w_2, yz_2w_1w_2\} \subseteq B$ has an empty intersection and thus we are done. 

Hence, we may assume without loss of generality that~$yz_1w_1w_2$ is red. Since $yw_1w_2w_3 \in H^{\red} \setminus R$ and $yz_1w_1 \in \partial B$, we have $yz_1z_2w_1, yz_1z_3w_1 \in B$.
If~$yz_2w_1w_2$ is red, then $yz_2z_3w_1 \in B$ (since $yw_1w_2w_3 \in H^{\red} \setminus R$ and $yz_2w_1 \in \partial B$). Thus we are done since $\{z_1z_2z_3u, yz_1z_2w_1, yz_1z_3w_1, yz_2z_3w_1\} \subseteq B$ has an empty intersection. 
If~$yz_2w_1w_2$ is blue, then we have $yz_2w_1w_2 \in B$ since $yw_1w_2 \in \partial B$.
Thus we are done since $\set{z_1z_2z_3u, yz_1z_3w_1, yz_2w_1w_2} \subseteq B$ has an empty intersection. 
\end{enumerate} 
This completes the proof.
\end{proof}

\subsection{\texorpdfstring{Proof of \cref{lem:increase_matching} assuming \cref{H2}}{Proof of Lemma 4.5 assuming (H2)}}

We now prove the remaining case of \cref{lem:increase_matching}, that is when~$M$ is contained in a blue tight component~$B$ of~$H$. Note that the proof is similar to the proof for the case where we assume \cref{H1}.

\begin{proof}[Proof of \cref{lem:increase_matching} assuming \cref{H2}] 
Assume for a contradiction that~$H$ does not contain a good $1/r$-fractional matching in~$R$ or in a blue tight component of~$H$ of weight at least $\s{M} + \gamma n$. 
Note that $\s{V(G)} \geq (1-\eps)N \geq N - 2\eps n \geq (5/4 +2\eta)n$. We will construct all our good fractional matchings in~$H[V(G)]$ ignoring the small number of vertices in $V(H) \setminus V(G)$.

It suffices to assume that~$M$ is a maximum good matching in~$B$, that is a maximum matching in~$B^+$. Let $W = V(G) \setminus V(M)$. Note that $B^+[W] = \varnothing$. 

\begin{claim}
\label{claim:B^2_in_f_2_exact}
If $f \in M$ is an edge such that $G^{\blue}[f]$ contains a triangle or a matching of size~$2$, then $G^{\blue}[f]$ contains an edge $e \in B^2$. Moreover, if $\s{G^{\blue}[f]} \geq 4$, then $G^{\blue}[f] \subseteq B^2$.
\end{claim}
\begin{proofclaim}
Let~$f$ be such an edge in~$M$. Since~$M$ is a good matching, there exists $z \in f$ such that $xyz \in \partial H(xy)$ for every $xy \in \binom{f \setminus \{z\}}{2}$. Observe that there exists $e \in \binom{f \setminus \{z\}}{2} \cap G^{\blue}$. 
Hence $e \cup z \in \partial B(e)$. Since $f \in B$, we have $B(e) = B$, that is, $e \in B^2$.

If $\s{G^{\blue}[f]} \geq 4$, then $G^{\blue}[f]$ contains a triangle or a matching of size~$2$ and thus by the previous argument $G^{\blue}[f]$ contains an edge $e \in B^2$. Moreover, $G^{\blue}[f]$ is connected and thus, since~$G$ is a blueprint, we have $G^{\blue}[f] \subseteq B^2$.
\end{proofclaim}

Let $U \subseteq W$ be a set of maximum size such that for each $u \in U$ there exists a distinct edge $f(u) \in M$ for which $(f(u), u)$ is a suitable pair and $H^{\blue}[f(u) \cup u] \cong K_5^{(4)}$. If $\s{U} \geq 4\gamma n$, then we are done since the $1/r$-fractional matching $\varphi \colon B \rightarrow [0,1]$ with $\varphi(e) = 1$ for $e \in M \setminus \{f(u) \colon u \in U\}$,  $\varphi(e) = 1/4$ for each edge $e \in \bigcup_{u \in U} H^{\blue}[f(u) \cup u]$ and $\varphi(e) = 0$ for all other edges $e \in B$ is a good $1/r$-fractional matching in~$B$ and has weight at least $\s{M} + \gamma n$. Now assume that $\s{U} < 4 \gamma n$. 

Let $W' = W \setminus U$ and $M' = M \setminus \{f(u) \colon u \in U\}$. Note that 
\begin{align}
\label{W'_size}
\s{W'} \geq (5/4 + 2\eta)n - 4\s{M} - 4\gamma n \geq (1/4 + \eta) n    
\end{align}
and $2 \delta n \leq \s{M} - 4\gamma n \leq \s{M'} \leq n/4$. 
By the maximality of~$U$, we have  
\begin{equation}
\label{Uproperty2_exact}
H^{\red}[f \cup w] \neq \varnothing
\end{equation}
for every suitable pair $(f, w) \in M' \times W'$. 

We distinguish between two cases.

\begin{enumerate}[label=\textbf{Case \arabic*:},ref = \arabic*, wide, labelwidth=0pt, labelindent=0pt]
\item \textbf{\boldmath $G^{\red}[W'] = \varnothing$ and $B^2[W'] \neq \varnothing$.\unboldmath}
Recall that $\s{W'} \geq (1/4 + \eta)n$ and $\delta(G) \geq (1-\eps)N$. So~$G[W']$ is connected. Hence $G[W'] \subseteq B^2$.
We start by proving the following claim.

\begin{claim}
\label{claim:in_R_*}
There exists a red tight component~$R_*$ of~$H$ such that $H^+[W'] \subseteq R_*^+$.
\end{claim}
\begin{proofclaim}
By the maximality of~$M$ and $G[W'] \subseteq B^2$, we have $H^+[W'] \subseteq H^{\red}$.\COMMENT{Suppose there was an edge $f \in H^+[W'] \cap H^{\blue}$. Since $f \in H^+$, by~\cref{G3}, there exists $z \in f$ such that $xyz \in \partial H(xy)$ for all $xy \in \binom{f \setminus \{z\}}{2}$. Let $xy \in \binom{f \setminus \{z\}}{2}$. Since $G[W'] \subseteq B^2$, we have $xy \in B^2$ and thus $xyz \in \partial B$. Hence $f \in B$, a contradiction to the maximality of~$M$.} By~\cref{cor:H^+_tightly_connected},~$H^+[W']$ is tightly connected. Hence there exists a red tight component~$R_*$ of~$H$ such that $H^+[W'] \subseteq R_*^+$.
\end{proofclaim}

Let $U' \subseteq W'$ be a set of maximum size such that  for each $u \in U'$, there exists a distinct edge $f'(u) \in M'$ so that $(f'(u), u)$ is a suitable pair and $R_*[f'(u) \cup u] \neq \varnothing$. Let $W'' = W' \setminus U'$ and $M'' = M' \setminus \{f'(u)\colon u \in U'\}$. Note that 
\[\s{W''} \geq \s{W'} -  \s{U'} \geq (1/4 + \eta)n - \s{M'} \geq \eta n.
\]
Let~$\delta_0$ be a new constant such that $\gamma \ll \delta_0 \ll \delta$.
\begin{claim}
\label{claim:U'_size_R_*}
We have $\s{U'} \geq \s{M'} -\delta_0 n \geq \s{M} -2\delta_0 n.$ 
\end{claim}
\begin{proofclaim}
Suppose not. We have $\s{M''} \geq \delta_0 n \geq \delta_0 N/2$. By the maximality of~$U'$, we have 
\begin{equation}
\label{U'property3_exact}
R_*[f \cup w] = \varnothing
\end{equation}
for every suitable pair $(f, w) \in M'' \times W''$. 
By~\cref{lem:make_suitable}, there exists $M^* \subseteq M''$ with $\s{M^*} = r \gamma n$ and disjoint sets $W_f \in \binom{W''}{4}$ for each $f \in M^*$ such that $(f, W_f)$ is a suitable pair for each $f \in M^*$.
Let~$\varphi_0$ be the fractional matching induced by the matching $M \setminus M^*$.
It suffices to show that, for every $f \in M^*$, there exists a good $1/r$-fractional matching~$\varphi_f$ in $B[f \cup W_f]$ of weight at least~$\frac{r+1}{r}$. Indeed the completion of $\varphi_0 + \sum_{f \in M^*} \varphi_f$ with respect to~$B$ is a good $1/r$-fractional matching in~$B$ of weight at least $\s{M \setminus M^*} + \frac{r+1}{r}\s{M^*} \geq \s{M} + \gamma n$ giving us a contradiction.

Consider any $f = x_1x_2x_3x_4 \in M^*$. By~\cref{fact:A}, $r = \binom{9}{4}!$ and $B[f \cup W_f] \subseteq B^+$, we may assume that $\bigcap B[f \cup W_f] \neq \varnothing$.

Let $W_f = w_1w_2w_3w_4 \in H^+[W''] \subseteq R_*^+$. Let $W_f^* = w_1w_2w_3$ and $e = w_1w_2$. Note that $(f, W_f^*)$ is a suitable pair and $e \in B^2$ since $G[W'] \subseteq B^2$.
We apply \cref{prop:local} with colours reversed and $r, H, G, B, f, W_f^*, e$ playing the roles of $r, H, G, R_*, f, W, e$. So there exists $x \in f$ such that 
\begin{align}
    \label{Local_application_2}
    \text{$f' \in H^{\red}$ for any edge $f' \in H[f \cup W_f^*]$ with $e = w_1w_2 \subseteq f'$ and $x \not\in f'$.}
\end{align}
By~\cref{Uproperty2_exact} and \cref{U'property3_exact}, there exists an edge $f_* \in H^{\red}[f \cup w_1] \setminus R_*$. Since $f \in B$, we have $w_1 \in f_*$. Let $yz = (f_* \cap f) \setminus x$. By~\cref{Local_application_2}, $yw_1w_2w_3, yzw_1w_2 \in H^{\red}$.
Since $w_1w_2w_3w_4 \in H^+[W''] \subseteq R_*^+$, $w_1w_2w_3 = W_f^* \in \partial R_*$. Hence $yw_1w_2w_3, yzw_1w_2, f_* \in R_*$, a contradiction to $f_* \not\in R_*$.
\end{proofclaim}

We now find a matching in~$R_*^+$ as follows.
For each $u \in U'$, choose an edge $f_u^* \in R_*[f'(u) \cup u]$ and note that, since $(f'(u), u)$ is a suitable pair,~$f_u^*$ is good. Let $M_1^* = \{f_u^* \colon u \in U'\}$, so~$M_1^*$ is a matching in $R_*^+[V(H) \setminus W'']$.
By~\cref{claim:in_R_*}, $H^+[W''] \subseteq R_*^+$.
By~\cref{prop:greedy_matching},~$R_*^+[W'']$ contains a matching~$M_2^*$ of size at least $\frac{\s{W''}}{4} - \gamma N \geq \frac{\s{W''}}{4} - 2\gamma n$. 

Thus $M_1^* \cup M_2^*$ is a matching in~$R_*^+$ of size
\begin{align*}
    \s{M_1^*} + \s{M_2^*} \crel{\geq} \s{U'} + \frac{\s{W''}}{4} -2\gamma n 
    = \frac{1}{4}(3\s{U'}+\s{W'}) -2\gamma n \\
    \crel{\overbracket[0pt][3pt]{\geq}^{\textup{Claim }\ref{claim:U'_size_R_*}, \, \cref{W'_size}}} \frac{1}{4}(3\s{M} -6 \delta_0 n + (5/4 + 2\eta)n - 4\s{M} - 4\gamma n) -2\gamma n\\
    \crel{\geq} \, \frac{1}{4}((5/4 + \eta)n - \s{M}) 
    \geq \frac{n}{4},
\end{align*}
where the last inequality holds as $\s{M} < n/4$.
Hence~$H$ contains a good monochromatic tightly connected matching of size at least~$n/4$, a contradiction.

\item \textbf{\boldmath $G^{\red}[W'] \neq \varnothing$ or $B^2[W'] = \varnothing$.\unboldmath}
Recall that by the maximality of~$M$, we have $B^+[W'] = \varnothing$.
\begin{claim}
\label{claim:R_on_B^2_exact}
Let $y_1y_2y_3y_4 \in H^+[W']$ with $y_1y_2 \in B^2$ and $y_1y_2y_3 \in \partial B$. Then $y_1y_2y_3y_4 \in R^+$.
\end{claim}
\begin{proofclaim}
If $B^2[W'] = \varnothing$, then this is vacuously true. Hence we may assume that $G^{\red}[W'] \neq \varnothing$.
Suppose to the contrary, that $y_1y_2y_3y_4 \not \in R^+$. Since $B^+[W'] = \varnothing$, we have $y_1y_2y_3y_4 \in H^{\red} \setminus R$.
Let $x_1x_2 \in G^{\red}[W']$.
By~\cref{prop:z_1z_2z_3}, there exist vertices $z_1, z_2, z_3 \in W'$ such that $H^+[y_1y_2y_3y_4z_1z_2z_3] \cong K_7^{(4)}$, $H^+[x_1x_2z_1z_2z_3] \cong K_5^{(4)}$, $y_1y_2z_1 \in \partial H(y_1y_2)$, $y_1z_1z_2 \in \partial H(y_1z_1)$, $z_1z_2z_3 \in \partial H(z_1z_2)$, $x_1z_1z_2 \in \partial H(x_1z_1)$ and $x_1x_2z_1 \in \partial H(x_1x_2)$.

Since $y_1y_2 \in B^2$, we have $y_1y_2z_1 \in \partial B$. 
Since $B^+[W'] = \varnothing$ and $y_1y_2y_3y_4 \in H^{\red}\setminus R$, we have $y_1y_2z_1z_2, y_1y_2y_3z_1 \in H^{\red} \setminus R$.
This implies that $y_1z_1 \in G^{\blue}$ (or else $y_1y_2z_1z_2 \in R$) and so $y_1z_1 \in B^2$. Thus $y_1z_1z_2 \in \partial B$ and $y_1z_1z_2z_3 \in H^{\red} \setminus R$ (or else $B^+[W'] \neq \varnothing$). 
Similarly, we deduce that $z_1z_2 \in G^{\blue}$ and so $z_1z_2 \in B^2$. 
Thus $z_1z_2z_3 \in \partial B$ and since $B^+[W'] = \varnothing$, $x_1z_1z_2z_3 \in H^{\red} \setminus R$. 
It follows that $x_1z_1 \in G^{\blue}$ and so $x_1z_1 \in B^2$.
Thus $x_1z_1z_2 \in \partial B$ and since $B^+[W'] = \varnothing$, $x_1x_2z_1z_2 \in H^{\red}$. Since $x_1x_2 \in G^{\red}$, we have $x_1x_2z_1 \in \partial R$ and thus $y_1y_2y_3y_4 \in R^+$, a contradiction.
\end{proofclaim}

Let $U' \subseteq W'$ be a set of maximum size such that for each $u \in U'$, there exists a distinct edge $f'(u) \in M'$ for which $(f'(u), u)$ is a suitable pair and $R[f'(u) \cup u] \neq \varnothing$. Let $W'' = W' \setminus U'$ and $M'' = M' \setminus \{f'(u) \colon u \in U'\}$. Note that $\s{W''} = \s{W'} -  \s{U'} \geq (1/4 + \eta)n - \s{M'} \geq \eta n \geq \eta N/2$. 
Let~$\delta_0$ be a new constant such that $\gamma \ll \delta_0 \ll \delta$.
\begin{claim}
We have $\s{U'} \geq \s{M'} - \delta_0 n$. 
\end{claim}
\begin{proofclaim}
Suppose not. We have $\s{M''} \geq \delta_0 n \geq \delta_0 N/2$.
By the maximality of~$U'$, we have 
\begin{equation}
\label{U'property2_exact}
R[f \cup w] = \varnothing
\end{equation}
for every suitable pair $(f, w) \in M'' \times W''$.
By~\cref{lem:make_suitable}, there exists $M^* \subseteq M''$ with $\s{M^*} = r \gamma n$ and disjoint sets $W_f \in \binom{W''}{4}$ for each $f \in M^*$ such that $(f, W_f)$ is a suitable pair for each $f \in M^*$.

Let~$\varphi_0$ be the fractional matching induced by the matching $M \setminus M^*$.
Suppose that, for every $f \in M^*$, there exists a good $1/r$-fractional matching~$\varphi_f$ in $B[f \cup W_f]$ of weight at least~$\frac{r+1}{r}$. Then the completion of $\varphi_0 + \sum_{f \in M^*} \varphi_f$ with respect to~$B$ is a good $1/r$-fractional matching in~$B$ of weight at least $\s{M \setminus M^*} + \frac{r+1}{r}\s{M^*} \geq \s{M} + \gamma n$.
Thus it suffices to show that, for every $f \in M^*$, there exists a $1/r$-fractional matching~$\varphi_f$ in $B[f \cup W_f]$ of weight at least~$\frac{r+1}{r}$. 

Consider any $f = x_1x_2x_3x_4 \in M^*$. By \cref{fact:A}, $B[f \cup W_f] \subseteq B^+$ and $r = \binom{9}{4}!$, it suffices to show that $\bigcap B[f \cup W_f] = \varnothing$. Let $uv \in G[W_f]$.
We distinguish between several cases.

\begin{enumerate}[label=\textbf{Case \Alph*:}, ref=\Alph*, wide, labelwidth=0pt, labelindent=0pt]
\item \textbf{\boldmath $\s{G^{\red}[f]} \geq 3$ with $G^{\red}[f] \ncong K_{1,3}$ or $G^{\red}[f]$ is a matching of size~$2$.\unboldmath}\label{3_R^2_case_exact} By~\cref{Uproperty2_exact}, there exists an edge $f_* \in H^{\red}[f \cup u]$.
Observe that~$f_*$ contains an edge $xy \in G^{\red}[f]$ and $u \in f_*$. By~\cref{SP3}, we have $xyu \in \partial H(xy) = \partial R$. Hence $f_* \in R$, a contradiction to \cref{U'property2_exact}.

\item \textbf{\boldmath $uv \in G^{\red}$, $\s{G^{\red}[f]} \geq 2$ and $\bigcap G^{\red}[f] \neq \varnothing$.\unboldmath}
Without loss of generality assume that $x_1x_2, x_1x_3 \in G^{\red}[f]$ and $x_2x_3, x_2x_4, x_3x_4 \in G^{\blue}$. By~\cref{claim:B^2_in_f_2_exact}, we have $B(x_2x_3) = B(x_2x_4) = B(x_3x_4) = B$.
By~\cref{Uproperty2_exact} and \cref{U'property2_exact}, there exists an edge $f_* \in H^{\red}[f \cup u] \setminus R$. Observe that $f_* = x_2x_3x_4u$ since, by~\cref{SP3}, $x_1x_2u, x_1x_3u \in \partial R$. Let $e_{f,1} = x_2x_3uv$, $e_{f,2} = x_2x_4uv$, $e_{f,3} = x_3x_4uv$. Note that, for all $i \in [3]$, we have $e_{f,i} \cap f \in B^2$ and $e_{f,i} \cap W_f = uv \in R^2$ and thus, by~\cref{SP3} and~\cref{SP4}, we have $e_{f,i} \in R \cup B$. Since $f_* \not\in R$, we have $e_{f,i} \in B$ for all $i \in [3]$.
We are done since $\{f, e_{f,1}, e_{f,2}, e_{f,3}\} \subseteq B$ has an empty intersection.

\item \textbf{\boldmath $uv \in G^{\red}$ and $G^{\blue}[f]$ contains a copy of~$C_4$.\unboldmath}
Without loss of generality assume that $x_1x_2, x_2x_3, x_3x_4, x_1x_4 \in G^{\blue}[f]$. By~\cref{claim:B^2_in_f_2_exact}, $x_1x_2, x_2x_3, x_3x_4, x_1x_4 \in B^2$. By~\cref{SP4}, we have $uvx_i \in \partial R$ for all $i \in [4]$. By~\cref{SP3}, we have $x_1x_2u, x_2x_3u, x_3x_4u, x_1x_4u \in \partial B$. 
If~$x_1x_2uv$ and~$x_3x_4uv$ are blue, then both are in~$B^+$ and together with~$f$ they form a set $F \subseteq B[f \cup W_f]$ with $\bigcap F = \varnothing$. So we may assume that~$x_1x_2uv$ is red and thus in~$R$. Similarly, we may assume that $x_1x_4uv \in R$.
By~\cref{Uproperty2_exact} and \cref{U'property2_exact}, there exists an edge $f_* \in H^{\red}[f \cup u] \setminus R$. Since $x_1x_2uv, x_1x_4uv \in R$, we have $f_* = x_2x_3x_4u$ and $x_2x_3uv, x_2x_4uv, x_3x_4uv \in B$. Thus we are done since $F= \{f, x_2x_3uv, x_2x_4uv, x_3x_4uv\}$ has an empty intersection.

\item \textbf{\boldmath $uv \in B^2$.\unboldmath}
Let $W_f^* = uvw \subseteq W_f$. Note that since $(f, W_f)$ is a suitable pair and $W_f^* \subseteq W_f$, we have that $(f, W_f^*)$ is a suitable pair. Suppose for a contradiction, that $\bigcap B[f \cup W_f] \neq \varnothing$.
We apply \cref{prop:local} with colours reversed and $r, H, G, B, f, W_f^*, e$ playing the roles of $r, H, G, R_*, f, W, e$. We have that there exists $x \in f$ such that 
\begin{align}
    \label{Local_application_3}
    \text{$f' \in H^{\red}$ for any edge $f' \in H[f \cup W_f^*]$ with $e = uv \subseteq f'$ and $x \not\in f'$.}
\end{align}
By~\cref{Uproperty2_exact} and \cref{U'property2_exact}, there exists an edge $f_* \in H^{\red}[f \cup u] \setminus R$. Since $f \in B$, we have $u \in f_*$. Let $yz \subseteq (f_* \cap f) \setminus x$. By~\cref{Local_application_3}, $yuvw, yzuv \in H^{\red}$.
Let $w' \in W_f \setminus uvw$. Since $uv \in B^2$, we have $uvw \in \partial B$. Since $(f, W_f)$ is a suitable pair, $uvww' \in H^+$. By~\cref{claim:R_on_B^2_exact}, we have $uvww' \in R^+$ and thus $uvw \in \partial R$. Hence $yuvw, yzuv, f_* \in R$, a contradiction to $f_* \not\in R$.

\item \textbf{\boldmath $uv \in G^{\blue}\setminus B^2$.\unboldmath}
If $\s{G^{\red}[f]} \geq 3$ with $G^{\red}[f] \ncong K_{1,3}$ or $G^{\red}[f]$ is a matching of size~$2$, then we are in Case~\ref{3_R^2_case_exact}.
Hence we may assume that $G^{\blue}[f]$ contains a triangle.
We assume without loss of generality that $x_1x_2, x_2x_3, x_1x_3 \in G^{\blue}$.
By~\cref{claim:B^2_in_f_2_exact}, we have $G^{\blue}[f] \subseteq B^2$. 
By~\cref{SP3}, we have $x_1x_2u, x_2x_3u, x_1x_3u \in \partial B$. Let $B_* = B(uv) \neq B$. By~\cref{SP4}, we have $uvx_i \in \partial B_*$ for all $i \in [3]$. Hence, since $B_* \neq B$, we have $E = \{x_1x_2uv, x_2x_3uv, x_1x_3uv\} \subseteq H^{\red}$.
Since $uv \in B_*^2$ and $B_* \neq B$, we have $x_iu \in G^{\red}$ for all $i \in [3]$. By~\cref{SP5}, we have $x_iuv \in \partial R$ for all $i \in [3]$. Hence $E \subseteq R$. It follows that $x_1x_2u, x_2x_3u, x_1x_3u \in \partial R$.
This contradicts the fact that $H^{\red}[f \cup u] \setminus R \neq \varnothing$ which holds by~\cref{Uproperty2_exact} and \cref{U'property2_exact}. \qedhere
\end{enumerate}
\end{proofclaim}

For the remainder of the proof, our aim is to find a good $1/r$-fractional matching in~$R$ of weight at least $\s{M} + \gamma n$.
For each $u \in U'$, choose an edge $f_u^* \in R[f'(u) \cup u]$. Since $(f'(u), u)$ is a suitable pair,~$f_u^*$ is good. Let $M_1^* = \{f_u^* \colon u \in U'\}$, so~$M_1^*$ is a good matching in~$R$. 
Note 
\[
\s{M_1^*} = \s{U'} \geq \s{M'} - \delta_0 n \geq \s{M} - 2\delta_0 n.
\]
Let $M_2^* \subseteq R^+$ be a maximum matching vertex-disjoint from~$M_1^*$. If $\s{M_1^*} + \s{M_2^*} \geq \s{M} + \gamma n$, then we are done. Thus we may assume $\s{M_1^*} + \s{M_2^*} < \s{M} + \gamma n$, so $\s{M_2^*} \leq 3 \delta_0 n$. Let $U'' = \{u \in U' \colon (f'(u) \cup u) \cap V(M_2^*) = \varnothing\}$. We have $\s{U''} \geq \s{U'} - 4\s{M_2^*} \geq \s{M} - 14\delta_0 n \geq 2\delta n \geq \delta N$. 
Let $M_0 = \{f'(u) \colon u \in U''\}$ and note that $\s{M_0} = \s{U''} \geq \delta N$.
Recall that $W'' = W' \setminus U'$ and $\s{W''} \geq \eta N/2$. Let $W_0 = W'' \setminus V(M_2^*)$ and note that $\s{W_0} \geq \s{W''} - 4 \s{M_2^*} \geq \eta N/4$.
By~\cref{lem:make_suitable}, there exist a subset $U_0 \subseteq U''$ corresponding to the matching $\{f'(u) \colon u \in U_0\} \subseteq M_0$ of size $3 r \delta_0 n$ and disjoint sets $W_u \in \binom{W_0}{4}$ for each $u \in U_0$ such that $(f'(u), W_u)$ is a suitable pair for each $u \in U_0$.

We now construct a good $1/r$-fractional matching $\varphi \colon R \rightarrow [0,1]$ in~$R$ as follows. 
Let~$\varphi_0$ be the fractional matching induced by the matching $(M_1^* \setminus \{f_u^* \colon u \in U_0\}) \cup M_2^*$. Suppose that, for each $u \in U_0$, there exists a good $1/r$-fractional matching~$\varphi_u$ in $R[f'(u) \cup u \cup W_u]$ of weight at least~$\frac{r+1}{r}$. Then the completion of $\varphi_0 + \sum_{u \in U_0} \varphi_u$ with respect to~$R$ is a good $1/r$-fractional matching in~$R$ of weight at least $\s{M_1^*} + \s{M_2^*} + \s{U_0}/r \geq \s{M} - 2\delta_0 n + 3 \delta_0 n \geq \s{M} + \gamma n$. Thus it suffices to show that, for each $u \in U_0$, there exists a good $1/r$-fractional matching~$\varphi_u$ in $R[f'(u) \cup u \cup W_u]$ of weight at least~$\frac{r+1}{r}$.

Consider any $u \in U_0$. Note that~$f_u^*$ is good and since $(f'(u), W_u)$ is a suitable pair, any edge in $H[f'(u) \cup W_u]$ is good.
By~\cref{fact:A}, it suffices to show that $\bigcap (R[f'(u) \cup W_u] \cup \{f_u^*\}) = \varnothing$.

Let 
\[
\text{$f'(u) = yz_1z_2z_3 \in B$, $f_u^* = z_1z_2z_3u \in R$ and $W_u = w_1w_2w_3w_4$.}
\]
By the maximality of~$M$, we have $w_1w_2w_3w_4 \not\in B$. By the maximality of~$M_2^*$, we have $w_1w_2w_3w_4 \not\in R$. Hence \cref{claim:R_on_B^2_exact} and~\cref{SP6} imply that $B^2[W_u] = \varnothing$. It follows that the following cases exhaust all possibilities.

\begin{enumerate}[label=\textbf{Case \Alph*:}, ref=\Alph*, wide, labelwidth=0pt, labelindent=0pt]
\item \textbf{\boldmath $G^{\red}[W_u] \neq \varnothing$.\unboldmath}\label{Case:w_1w_2_red_exact}
We assume without loss of generality that $w_1w_2 \in G^{\red}$.
By~\cref{SP6} and~\cref{SP4}, $w_1w_2w_3, w_1w_2y \in \partial R$. Since $w_1w_2w_3w_4 \not\in R \cup B$, we have $w_1w_2w_3w_4 \in H^{\blue} \setminus B$. By the maximality of~$M_2^*$, we have $yw_1w_2w_3 \in H^{\blue}\setminus B$. We now consider the colours of the edges~$yz_i$ for $i \in [3]$.

\begin{enumerate}[label=\textbf{Case  \Alph{enumii}.\arabic*:},ref = \Alph{enumii}.\arabic*, wide, labelwidth=0pt, labelindent=0pt]
\item \textbf{\boldmath At least two edges in $\{yz_i \colon i \in [3]\}$ are in~$B^2$.\unboldmath}
We assume without loss of generality that $yz_1, yz_2 \in B^2$. By~\cref{SP3}, $yz_1w_1, yz_2w_1 \in \partial B$. Since $yw_1w_2w_3 \in H^{\blue}\setminus B$, we have $yz_1w_1w_2, yz_2w_1w_2 \in H^{\red}$. Since $w_1w_2y \in \partial R$, we have $\{z_1z_2z_3u, yz_1w_1w_2, yz_2w_1w_2\} \subseteq R$. Moreover, this set has an empty intersection and so we are done.

\item \textbf{\boldmath At least two edges in $\{yz_i \colon i \in [3]\}$ are in~$G^{\red}$.\unboldmath}
We assume without loss of generality that $yz_1, yz_2 \in G^{\red}$. 
By~\cref{SP3}, we have $yz_1w_1, yz_2w_1 \in \partial R$.
We distinguish between the following three subcases.

If $yz_1w_1w_2, yz_2w_1w_2$ are both red, then, $yz_1w_1w_2, yz_2w_1w_2 \in R$ as $yw_1w_2 \in \partial R$. We are done since $\{z_1z_2z_3u, yz_1w_1w_2, yz_2w_1w_2\} \subseteq R$ has an empty intersection.

If~$yz_1w_1w_2$ is blue and~$yz_2w_1w_2$ is red, then
since $yz_2w_1 \in \partial R$ and $yw_1w_2w_3 \in H^{\blue} \setminus B$, we have $yz_2w_1w_2 \in R$ and $yz_1w_1w_2 \in H^{\blue} \setminus B$. From $yz_1z_2z_3 \in B$ and $yz_1w_1w_2 \in H^{\blue} \setminus B$, it follows that $yz_1z_3w_1 \in H^{\red}$. Since $yz_1w_1 \in \partial R$, we have $yz_1z_3w_1 \in R$. We are done since $\{z_1z_2z_3u, yz_1z_3w_1, yz_2w_1w_2\} \subseteq R$ has an empty intersection.

Hence we may assume that $yz_1w_1w_2, yz_2w_1w_2$ are both blue.
Since $yw_1w_2w_3 \in H^{\blue} \setminus B$ and $yz_1z_2z_3 \in B$, we have $yz_1z_2w_1, yz_1z_3w_1, yz_2z_3w_1 \in H^{\red}$. From $yz_1w_1, yz_2w_1 \in \partial R$, it follows that $yz_1z_2w_1, yz_1z_3w_1, yz_2z_3w_1 \in R$. We are done since $\set{z_1z_2z_3u, yz_1z_2w_1, yz_1z_3w_1, yz_2z_3w_1} \subseteq R$ has an empty intersection.

\item \textbf{\boldmath At least two edges in $\{yz_i \colon i \in [3]\}$ are in $G^{\blue} \setminus B^2$.\unboldmath}
We assume without loss of generality that $yz_1, yz_2 \in G^{\blue} \setminus B^2$. Let $B_* = B(yz_1) = B(yz_2)$ and note that $B_* \neq B$. We have $z_1z_2, z_1z_3, z_2z_3 \in G^{\red}$ (or else \cref{claim:B^2_in_f_2_exact} implies $yz_1, yz_2 \in B^2$). Since $yz_1w_1, yz_2w_1 \in \partial B_*$ by~\cref{SP3} and $yz_1z_2z_3 \in B \neq B_*$, we have  $yz_1z_2w_1, yz_1z_3w_1, yz_2z_3w_1 \in H^{\red}$. Since $z_1z_2w_1, z_1z_3w_1, z_2z_3w_1 \in \partial R$ by~\cref{SP3}, we have $yz_1z_2w_1, yz_1z_3w_1, yz_2z_3w_1 \in R$. We are done since $\set{z_1z_2z_3u, yz_1z_2w_1, yz_1z_3w_1, yz_2z_3w_1} \subseteq R$ has an empty intersection.
\end{enumerate}  

\item \textbf{\boldmath $G[W_u] \subseteq G^{\blue}\setminus B^2$.\unboldmath}\label{Case:w_1w_2_blue}
Since~$G$ is a blueprint all the edges in~$G[W_u]$ induce the same blue tight component $B_* \neq B$ of~$H$. By~\cref{SP4} and~\cref{SP6}, $\listing{yw_1w_2, yw_1w_3, yw_2w_3, w_1w_2w_3} \in \partial B_*$. 

\begin{enumerate}[label=\textbf{Case  \Alph{enumii}.\arabic*:},ref = \Alph{enumii}.\arabic*, wide, labelwidth=0pt, labelindent=0pt]
\item \textbf{\boldmath At least one edge in $\{yz_i \colon i \in [3]\}$ is in~$B^2$.\unboldmath}
\label{Case:yz_i_in_B^2_exact}We assume without loss of generality that $yz_1 \in B^2$. By~\cref{SP4}, $yz_1w_1 \in \partial B$. Note that $yw_1 \in G^{\red}$ (else $B_* = B$ since~$G$ is a blueprint). By~\cref{SP5}, $yw_1w_2 \in \partial R$ and the maximality of~$M_2^*$ implies $yw_1w_2w_3 \in B_*$.
Since $B\neq B_*$, $yz_1w_1 \in \partial B$ and $yw_1w_2 \in \partial R$, we have $yz_1w_1w_2 \in R$. 

If~$yz_2w_1w_2$ is red, then we have $yz_2w_1w_2 \in R$ as $yw_1w_2 \in \partial R$. Moreover, $\set{z_1z_2z_3u, yz_1w_1w_2, yz_2w_1w_2} \subseteq R$ has an empty intersection. Hence we may assume that~$yz_2w_1w_2$ is blue.
We have $yz_2w_1w_2 \in B_*$ since $yw_1w_2w_3 \in B_*$. It follows that $yz_2z_3w_1 \in H^{\red}$ (else $B = B_*$). 

Now if $yz_2 \in G^{\red}$, then $yz_2w_1 \in \partial R$ by~\cref{SP3} and thus $yz_2z_3w_1 \in R$. We are done since $\{z_1z_2z_3u, yz_1w_1w_2, yz_2z_3w_1\} \subseteq R$ has an empty intersection. Hence we may assume that $yz_2 \in G^{\blue}$. 
Since $yz_1 \in B^2$ and~$G$ is a blueprint, we have $yz_2 \in B^2$. By~\cref{SP3}, we have $yz_2w_1 \in \partial B$, a contradiction to $yz_2w_1w_2 \in B_*$ since $B_* \neq B$.

\item \textbf{\boldmath At least one of the edges $z_1z_2, z_1z_3, z_2z_3$ is in~$B^2$.\unboldmath}
We assume without loss of generality that $z_1z_2 \in B^2$. We may assume that $yz_1, yz_2 \in G^{\red}$ (else we are in Case~\ref{Case:yz_i_in_B^2_exact}). By~\cref{SP4}, we have $yz_1w_1,yz_2w_1 \in \partial R$. Let $F_1 = \{yz_1w_1w_2, yz_1z_3w_1\}$ and $F_2 = \{yz_2w_1w_2, yz_2z_4w_1\}$. We claim that each of~$F_1$ and~$F_2$ contains a red edge. Suppose not, and assume without loss of generality that $F_1 \subseteq H^{\blue}$. Since $yz_1z_2z_3 = f'(u) \in B$ and $yw_1w_2 \in \partial B_*$, we have $B_* = B$, a contradiction. Let $f_1 \in F_1$ and $f_2 \in F_2$ be red edges. Since $yz_1w_1,yz_2w_1 \in \partial R$, we have $f_1, f_2 \in R$. We are done since $\{f_u^*, f_1, f_2\} \subseteq R$ has an empty intersection.

\item  \textbf{\boldmath~$f'(u)$ contains no edges of~$B^2$.\unboldmath}
Since~$f'(u)$ contains no edges of~$B^2$, \cref{claim:B^2_in_f_2_exact} implies that $G^{\blue}[f'(u)]$ does not contain a triangle.
Thus we may choose edges $e_{12} \in G^{\red}[yz_1z_2],\allowbreak e_{13} \in G^{\red}[yz_1z_3]$ and $e_{23} \in G^{\red}[yz_2z_3]$. Let $F_{12} = \{yz_1z_2w_1, e_{12} \cup w_1w_2\}, \allowbreak F_{13} = \{yz_1z_3w_1, e_{13} \cup w_1w_3\}$ and $F_{23} = \{yz_2z_3w_2, e_{23} \cup w_2w_3\}$. 
Suppose that each of $F_{12}, F_{13}$ and~$F_{23}$ contains a red edge $f_{12}, f_{13}$ and~$f_{23}$, respectively.
By~\cref{SP3}, we have that $e_{12} \cup w_1, e_{13} \cup w_1, e_{23} \cup w_2 \in \partial R$ and thus $F = \{f_u^*, f_{12}, f_{13}, f_{23}\} \subseteq R$. We are done since~$F$ has an empty intersection.
Hence we may assume that one of $F_{12}, F_{13}$ and~$F_{23}$ contains only blue edges. 
We assume without loss of generality that~$F_{12}$ contains only blue edges. That is~$yz_1z_2w_1$ and $e_{12} \cup w_1w_2$ are blue. Note that these two edges are in~$B$ since $yz_1z_2z_3 = f'(u) \in B$.
By~\cref{SP4}, we have $z_1w_1w_2, z_2w_1w_2 \in \partial B_*$.
Hence $yz_1z_2w_1, e_{12} \cup w_1w_2 \in B_*$. This contradicts $B_* \neq B$. 
\end{enumerate}
\end{enumerate} 
\end{enumerate}
This completes the proof.
\end{proof}

\section{\texorpdfstring{Proof of \cref{thm:main}}{Proof of the Main Theorem}}
\label{section:proof_of_thm}

\begin{defn}
Let $\mu_k^1(\beta,\eps, n)$ be the largest~$\mu$ such that every $2$-edge-coloured $(1-\eps, \eps)$-dense $k$-graph on~$n$ vertices contains a fractional matching with weight~$\mu$ such that all edges with non-zero weight have weight at least~$\beta$ and lie in a single monochromatic tight component.
Let $\mu_k^1(\beta) = \liminf_{\eps \to 0} \liminf_{n \to \infty} \mu_k^1(\beta,\eps,n)/n$.
\end{defn}

The following is the crucial result that reduces finding cycles in the original graph to finding tightly connected matchings in the reduced graph. 

\begin{cor}[Corollary 20 in \cite{Lo2020}]
\label{cor:matchings_to_cycles}
Let $1/n \ll \eta, \beta, 1/k$ with $k \ge 3$. 
Let~$K$ be a $2$-edge-coloured complete $k$-graph on~$n$ vertices.
Then~$K$ contains a monochromatic tight cycle of length~$\ell$ for any $\ell \leq (\mu_k^1(\beta) - \eta)k n$ divisible by~$k$. 
\end{cor}

We are now ready to prove \cref{thm:main}.

\begin{proof}[Proof of \cref{thm:main}]
Let $1/n \ll c \ll \eta \ll \eps$. Let~$K$ be a $2$-edge-coloured complete $4$-graph on $N = (5 + \eps)n$ vertices. We show that~$K$ contains a monochromatic tight cycle of length~$4n$. Note that \cref{lem:main_matchings_lemma} implies that $\mu_4^1(c) \geq 1/5 - \eta$. Applying \cref{cor:matchings_to_cycles} with $N, \eta, c, 4, K$ playing the roles of $n, \eta, \beta, k, K$ we obtain that~$K$ contains a monochromatic tight cycle of length~$\ell$ for any 
$\ell \leq (\mu_4^1(c) - \eta)4N$ divisible by~$4$. 
Since 
\[
(\mu_4^1(c) - \eta)4N \geq (1/5 - 2\eta)(5 +  \eps)4n \geq 4n,
\]
we have that~$K$ contains a monochromatic tight cycle of length~$4n$.
\end{proof}

\section{Concluding Remarks}

Here we determined the Ramsey number for $4$-uniform tight cycles asymptotically in the case where the length of the tight cycle is divisible by $4$. The cases where the length of the tight cycle is not divisible by $4$ are still open.
The general conjecture for the Ramsey numbers of tight cycles is as follows.
\begin{conj}[Haxell, {\L}uczak, Peng, R\"odl, Ruci\'nski, Skokan \cite{Haxell2009}] \label{conj:general}
Let $k \geq 2$, $0 \leq i \leq k-1$ and $d = \gcd(k,i)$. Then $r(C_{kn+i}^{(k)}) = (1+o(1)) \frac{d+1}{d}kn$.
\end{conj}
The lower bound is given by the following extemal example.
\begin{prop}
Let $n \geq 1, k\geq 2$ and $0 \leq i \leq k-1$. Then $r(C_{kn+i}^{(k)}) \geq \frac{d+1}{d}kn-2$ where $d = \gcd(k,i)$.
\end{prop}
\begin{proof}
Let $N = \frac{d+1}{d}kn-2$ and consider the following red-blue edge-colouring of $K_N^{(k)}$. Partition the vertex set of $K_N^{(k)}$ into two sets $X$ and $Y$ such that $\s{X} = \frac{k}{d}n-1$ and $\s{Y} = kn -1$. Colour each edge that has an even number of vertices in $X$ red and all other edges blue. 
Note that each monochromatic tight component in this red-blue edge-colouring of $K_N^{(k)}$ consists of all edges $e$ such that $\s{e \cap X} = r_1$ and $\s{e \cap Y} =r_2$ for some pair of nonnegative integers $(r_1, r_2)$ with $r_1 + r_2 = k$. We claim that this red-blue edge-colouring of $K_N^{(k)}$ does not contain a monochromatic copy of $C_{kn+i}^{(k)}$. Suppose for a contradiction that there is a monochromatic copy $C$ of $C_{kn+i}^{(k)}$. Let $(r_1, r_2)$ be the pair of nonnegative integers that correspond to the monochromatic tight component that contains $C$. Note that if $r_1 = 0$, then $V(C) \subseteq Y$. But since $\s{V(C)} = kn + i > kn - 1 = \s{Y}$, this is impossible. Hence $r_1 \geq 1$.

First suppose that $i = 0$. Then $d = k$ and so $\s{X} = n-1$. By double counting the pairs $(v, e)$ such that $v \in V(C) \cap X$ and $v \in e \in C$, we have $k \s{V(C) \cap X} = r_1 kn$. Since $n-1 = \s{X} \geq \s{V(C) \cap X} = r_1n \geq n$, we have a contradiction.

Now suppose that $1 \leq i \leq k-1$. By the same double counting argument as above, we have $k\s{V(C) \cap X} = r_1(kn +i)$. Hence $k \mid r_1(kn +i)$ and thus $k \mid r_1 i$. Thus $r_1 i$ is a common multiple of $i$ and $k$. It follows that $ik = \gcd(k,i) \mathrm{lcm}(k,i) \leq d r_1 i$ and so $r_1 \geq \frac{k}{d}$. Now we have $\s{X} \geq \s{X \cap V(C)} = \frac{r_1}{k} (kn +i) \geq r_1 n \geq \frac{k}{d}n > \s{X}$, a contradiction.
\end{proof}

For $4$-uniform tight cycles, \cref{conj:general} implies that $r(C^{(4)}_{4n+1}) = (1+o(1))8n = r(C^{(4)}_{4n+3})$ and $r(C^{(4)}_{4n+2}) =(1+o(1)) 6n$.
In order to prove these remaining cases, finding a large monochromatic tightly connected fractional matching in the reduced graph is no longer sufficient. 
Indeed, if the corresponding monochromatic tight component in the original graph is a complete $4$-partite 4-graph, then it only contains tight cycles of length divisible by $4$. 
A natural approach to overcome this problem is to additionally require that the chosen monochromatic tight component in the reduced graph contains a copy of $C_5^{(4)}$ (or a subgraph homomorphic to $C_5^{(4)}$). One of the difficulties with this approach is that we can no longer just choose a maximum matching in a monochromatic tight component, as these matchings can now be arbitrarily large (as long as we cannot also find a subgraph homomorphic to $C_5^{(4)}$). Thus a lot of the arguments we used in our proof do no longer apply in this setting.

Moreover, there are also difficulties to generalising our approach to higher uniformities. In the $4$-uniform case the blueprint is a $2$-edge-coloured almost complete $2$-graph and so has a monochromatic (tight) component that contains almost all edges of one colour (see \cref{lem:generalblueprint,cor:Ecomp}). Since blueprints for $2$-edge-coloured $k$-graphs are $(k-2)$-graphs, for $k \geq 5$, we no longer have this fact (there are $2$-edge-coloured complete $3$-graphs with $2$ red and $2$ blue tight components each containing at least a $1/8$-fraction of the edges).

Nevertheless we hope that some of our methods will be useful for further research on \cref{conj:general}.

\bibliographystyle{abbrv}
\bibliography{bibliography}

\end{document}